\documentclass{article}
\usepackage{amsmath,amssymb,amsfonts,amsthm}
\usepackage{graphicx,float}
\usepackage{array}
\usepackage{todonotes}
\usepackage{tikz}
\usepackage{youngtab}
\usepackage[margin=30mm]{geometry}

\usetikzlibrary{arrows}
\usetikzlibrary{matrix}
\usetikzlibrary{decorations.pathreplacing}
\usetikzlibrary{positioning}

\newtheorem{lem}{Lemma}[section]
\newtheorem{thm}[lem]{Theorem}
\newtheorem{prop}[lem]{Proposition}
\newtheorem{cor}[lem]{Corollary}

\theoremstyle{remark}
\newtheorem{rem}[lem]{Remark}
\newtheorem*{rem*}{Remark}

\theoremstyle{definition}
\newtheorem{defn}[lem]{Definition}
\newtheorem{example}[lem]{Example}
\newtheorem*{notn*}{Notation}

\numberwithin{equation}{section}

\newcommand{\bbf}{\mathbb{F}}
\newcommand{\catmod}{\textbf{mod}}
\newcommand{\sym}{\mathfrak{S}}
\newcommand{\half}{\frac{1}{2}}

\newcommand{\ind}{\mathrm{ind}}

\title{On the decomposition matrix of the partition algebra in positive characteristic}
\author{Oliver King}
\date{}
\begin{document}

\maketitle
\begin{abstract}
	We examine the structure of the partition algebra $P_n(\delta)$ over a field $k$ of characteristic $p>0$. In particular, we describe the decomposition matrix of $P_n(\delta)$ when $n<p$ and when $n=p$ and $\delta=p-1$.
\end{abstract}

\section{Introduction}
	The partition algebra was originally defined by Martin in \cite{martin1994temperley} over $\mathbb{C}$ as a generalisation of the Temperley-Lieb algebra for $\delta$-state $n$-site Potts models in statistical mechanics, and independently by Jones \cite{jones1994potts}. Although this interpretation requires $\delta$ to be integral, it is possible to define the algebra for any $\delta$. It was shown in \cite{xi1999partition} that the partition algebra $P_n(\delta)$ over an arbitrary field $\mathbb{F}$ is a cellular algebra, with cell modules $\Delta_\lambda(n)$ indexed by partitions $\lambda$ of size at most $n$. If we suppose $\delta\neq0$, then in characteristic zero these partitions also label a complete set of non-isomorphic simple modules, given by the heads of the corresponding cell modules. In positive characteristic the simple modules are indexed by the subset of $p$-regular partitions (again under the assumption $\delta\neq0$). It is natural to then ask how the simple modules arise as composition factors of the cell modules. In the case $\mathrm{char}\;\mathbb{F}=0$ this has been entirely resolved by Martin \cite{martin1996structure} and Doran and Wales \cite{doran2000partition}, however there has previously been little investigation into the positive characteristic case. 

	Martin provides in \cite{martin1996structure} a condition on $\lambda$, $\mu$ and $\delta$ for when there is a homomorphism in characteristic zero between cell modules labelled by $\lambda$ and $\mu$, provided $\delta\neq0$. This was strengthened in \cite{doran2000partition} to allow for $\delta=0$. In \cite{bowmanblocks} this condition was reformulated in terms of the reflection geometry of a Weyl group $W$ under a $\delta$-shifted action. By then considering the action of the corresponding affine Weyl group $W^p$, a description of the blocks of the partition algebra in positive characteristic was given.\\~
		
	In this paper we continue to investigate the representations of $P_n^\bbf(\delta)$ when $\mathrm{char}\;\mathbb{F}=p>2$. We show that by placing certain restrictions on the values of $n$, $\delta$ and $p$ we can give a complete description of the decomposition matrices.
	
	In Section 2 we set up the notation and definitions that will be used throughout the paper, and review some previous results. In Section 3 we recall some results regarding the representation theory of the symmetric group, and the abacus method of representing partitions. Section 4 introduces the partition algebra and recalls the block structure in characteristic zero and in prime characteristic. In Section 5 we obtain the decomposition matrix of the partition algebra in positive characteristic, and is separated into three subsections, each dealing with a particular case of the values of $n$ and $\delta$.\\~
	
	When writing this paper, it was brought to the author's attention that the decomposition numbers of the partition algebra $P_n^k(\delta)$ for $n<p$ were obtained independently, and by different methods, by A. Shalile \cite{shalile2014modular}.
	
\subsection*{Notation}
Throughout this paper, we fix a prime number $p>2$ and a $p$-modular system $(K,R,k)$. That is, $R$ is a discrete valuation ring with maximal ideal $P=(\pi)$, field of fractions $\mathrm{Frac}(R)=K$ of characteristic 0, and residue field $k=R/P$ of characteristic $p$. We will use $\bbf$ to denote either $K$ or $k$.

We also fix a parameter $\delta\in R$ and assume that its image in $k$ is non-zero. We will use $\delta$ to denote both the element in $R$ and its projection in $k$.
	
\section{Preliminaries}

Suppose $A$ is an $R$-algebra, free and of finite rank as an $R$-module. We can extend scalars to produce the $K$-algebra \mbox{$KA=K\otimes_RA$} and the $k$-algebra \mbox{$kA=k\otimes_RA$}. Given an $A$-module $M$, we can then also consider the $KA$-module $KM=K\otimes_RM$ and the $kA$-module $kM=k\otimes_RM$.

The following lemma shows that we can reduce $K$-module homomorphisms to $k$-module homomorphisms.
	\begin{lem}
		Suppose $X,Y$ are $R$-free $A$-modules of finite rank and let $M\subseteq KY$. If\linebreak \mbox{$\mathrm{Hom}_{K}(KX,KY/M)\neq0$} then there is a submodule $N\subseteq kY$ such that $\mathrm{Hom}_{k}(kX,kY/N)\neq0$. Moreover, $N$ is the $p$-modular reduction of a lattice in $M$.
		\label{lem:modred}
	\end{lem}
	
		\begin{proof}
			Let \mbox{$Q=KY/M$} be the image of the canonical quotient map
			\linebreak
			\mbox{$\rho:KY\rightarrow KY/M$}, and let 
			$f\in\mathrm{Hom}_{K}(KX,Q)$ be non-zero. Note that $\rho(Y)$ is a lattice in $Q$, since $Y$ has finite rank and \mbox{$K\rho(Y)=\rho(KY)=Q$}. As $X$ and $Y$ are modules of finite rank we may assume that 
				\begin{equation}
					f(X)\subseteq \rho(Y)\text{~~~but~~~}f(X)\nsubseteq \pi \rho(Y)
					\label{eqn:pires}
				\end{equation}
for instance by considering the matrix of $f$ and multiplying the coefficients by an appropriate power of $\pi$.\\
Then $f$ restricts to a homomorphism $X\rightarrow\rho(Y)$, and induces a homomorphism $\overline{f}:kX\rightarrow k\rho(Y)$. This must be non-zero since we can find $x\in X$ such that $f(x)\in \rho(Y)\backslash\pi \rho(Y)$ by \eqref{eqn:pires}.\\

			It remains to prove that $k\rho(Y)$ can be taken to be $kY/N$ for some $N\subset kY$, the modular reduction of a lattice in $M$. We have the following maps:
				\begin{figure}[H]				
				\centering
					\begin{tikzpicture}[>=angle 60]
						\draw (0,3) node {0};
						\draw [->] (0.25,3)--(1.25,3);
						\draw (1.5,3) node {$M$};
						\draw [->] (1.75,3)--(2.75,3);
						\draw (3.15,3) node {$KY$};
						\draw [->>] (3.5,3)--node[above] {\small$\rho$}(4.5,3);
						\draw (4.75,3) node {$Q$};
						\draw [<-] (5,3)--node[above] {\small$f$}(6,3);
						\draw (6.45,3) node {$KX$};
						
						\draw (3.15,2.5) node {$\cup$};
						\draw (4.75,2.5) node {$\cup$};
						\draw (6.45,2.5) node {$\cup$};
						
						\draw (0,2) node {0};
						\draw [->] (0.25,2)--(1.25,2);
						\draw (1.5,2) node {$L$};
						\draw [->] (1.75,2)--(2.9,2);
						\draw (3.15,2) node {$Y$};
						\draw [->>] (3.35,2)--(4.3,2);
						\draw (4.75,2) node {$\rho(Y)$};
						\draw (6.45,2) node {$X$};
						
						\draw [->>] (3.15,1.75)--(3.15,1);
						\draw [->>] (4.75,1.75)--(4.75,1);
						\draw [->>] (6.45,1.75)--(6.45,1);

						\draw (3.15,0.75) node {$kY$};
						\draw [->>] (3.4,0.75)--(4.2,0.75);
						\draw (4.75,0.75) node {$k\rho(Y)$};
						\draw [<<-] (5.3,0.75)--node[above] {\small$\bar f$}(6.1,0.75);
						\draw (6.45,0.75) node {$kX$};
					\end{tikzpicture}
				\end{figure}
			where $L=\mathrm{Ker}(Y\longrightarrow\rho(Y))$\\
			\\
			The $K$-module $Q$ is torsion free. Therefore as an $R$-module, $\rho(Y)\subseteq Q$ must also be torsion free. Since $R$ is a principal ideal domain (by definition of it being a discrete valuation ring), the structure theorem for modules over a PID tells us that $\rho(Y)$ must be free. It is therefore projective, and the exact sequence
				\[0\longrightarrow L\longrightarrow Y\longrightarrow \rho(Y)\longrightarrow0\]
			is split.
			Then since the functors $K\otimes_R-$ and $k\otimes_R-$ preserve split exact sequences, we deduce that $M\cong KL$ and we can set $N=kL$ to complete the exact sequence
				\[0\longrightarrow N\longrightarrow kY\longrightarrow k\rho(Y)\longrightarrow0\]
satisfying the requirements above.
		\end{proof}

\section{Representation theory of the symmetric group}\label{sec:symmetricgroup}
A more detailed account of the results in this section can be found in \cite{james1981representation}.

\subsection*{Partitions}

For any natural number $n$, we define a partition $\lambda=(\lambda_1,\lambda_2,\dots)$ to be a weakly decreasing sequence of non-negative integers such that $\sum_{i\geq1}\lambda_i=n$. These conditions imply that $\lambda_i=0$ for $i\gg0$, hence we will often truncate the sequence and write $\lambda=(\lambda_1,\dots,\lambda_l)$, where $\lambda_l\neq0$ and $\lambda_{l+1}=0$. We also combine repeated entries and use exponents, for instance the partition $(5,5,3,2,1,1,0,0,0,\dots)$ of $17$ will be written $(5^2,3,2,1^2)$. We use the notation $\lambda\vdash n$ to mean $\lambda$ is a partition of $n$.
We let $\Lambda_n$ be the set of all partitions of $n$, and define the following set
	\begin{align}\label{eq:partsets}
		\Lambda_{\leq n}&=\bigcup_{0\leq i \leq n}\Lambda_n
	\end{align}

We say that a partition $\lambda=(\lambda_1,\dots,\lambda_l)$ is \emph{$p$-singular} if there exists $t$ such that
	\[\lambda_t=\lambda_{t+1}=\dots=\lambda_{t+p-1}>0\]
i.e. some (non-zero) part of $\lambda$ is repeated $p$ or more times.
Partitions that are not $p$-singular we call \emph{$p$-regular}. We let $\Lambda_n^\ast$ be the subset of $\Lambda_n$ of all $p$-regular partitions of $n$, and similarly define the set 
	\begin{align}
		\Lambda^\ast_{\leq n}&=\bigcup_{0\leq i \leq n}\Lambda^\ast_n\label{eq:regpartsets}
	\end{align}

There exists a partial order on the set $\Lambda_{\leq n}$ called the \emph{dominance order with size}, denoted by $\leq_d$. We say a partition $\lambda$ is less than or equal to $\mu$ under this order if either $|\lambda|<|\mu|$, or $|\lambda|=|\mu|$ and $\displaystyle\sum_{i=1}^j\lambda_i\leq\sum_{i=1}^j\mu_i$ for all $j\geq1$. We write $\lambda<_d\mu$ to mean $\lambda\leq_d\mu$ and $\lambda\neq\mu$. 

To each partition $\lambda$ we may associate the Young diagram 
	\[[\lambda]=\{(x,y)~|~x,y\in\mathbb{Z},~1\leq x\leq l,~1\leq y \leq \lambda_x\}\]
An element $(x,y)$ of $[\lambda]$ is called a \emph{node}. If $\lambda_{i+1}<\lambda_{i}$, then the node $(i,\lambda_i)$ is called a \emph{removable} node of $\lambda$. If $\lambda_{i-1}>\lambda_i$, then we say the node $(i,\lambda_i+1)$ of $[\lambda]\cup\{(i,\lambda_i+1)\}$ is an \emph{addable} node of $\lambda$. This is illustrated in Figure \ref{fig:youngdiag} below. If a partition $\mu$ is obtained from $\lambda$ by removing a removable (resp. adding an addable) node then we write $\mu\triangleleft\lambda$ (resp. $\mu\triangleright\lambda$).
	
Each node $(x,y)$ of $[\lambda]$ has an associated integer, called the \emph{content}, given by $y-x$.

	\begin{figure}[H]
		\centering
		\begin{tikzpicture}[scale=0.5]
			\foreach \x in {0,...,4}
				{\draw (\x,5) rectangle +(1,-1);}
			\foreach \x in {0,...,4}
				{\draw (\x,4) rectangle +(1,-1);}
			\foreach \x in {0,...,2}
				{\draw (\x,3) rectangle +(1,-1);}
			\draw (0,2) rectangle +(1,-1);
			\draw (1,2) rectangle +(1,-1);
			\draw (0,1) rectangle +(1,-1);
			\draw (0,0) rectangle +(1,-1);
			
			\draw (4.5,3.45) node {$r$};
			\draw (2.5,2.45) node {$r$};
			\draw (1.5,1.45) node {$r$};
			\draw (0.5,-0.55) node {$r$};

			\draw (5.5,4.45) node {$a$};
			\draw (3.5,2.45) node {$a$};
			\draw (2.5,1.45) node {$a$};
			\draw (1.5,0.45) node {$a$};
			\draw (0.5,-1.55) node {$a$};
		\end{tikzpicture}
		\caption{The Young diagram of $\lambda=(5^2,3,2,1^2)$. Removable nodes are marked by $r$ and addable nodes by $a$.}
		\label{fig:youngdiag}
	\end{figure}
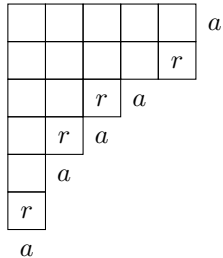

\subsection*{Abacus}

We recall the abacus method of constructing partitions from \cite[Chapter 2.7]{james1981representation}. To each partition and prime number $p$ we assoiate an abacus diagram, consisting of $p$ columns, known as runners, and a configuration of beads across these. By convention we label the runners from left to right, starting with 0, and the positions on the abacus are also numbered from left to right, working down from the top row, starting with $0$ (see Figure \ref{fig:symabacus}). Given a partition $\lambda=(\lambda_1,\dots,\lambda_l)\vdash n$, fix a positive integer $b\geq n$ and construct the $\beta$-sequence of $\lambda$, defined to be
	\[\beta(\lambda,b)=(\lambda_1-1+b,\lambda_2-2+b,\dots,\lambda_l-l+b,-(l+1)+b,\dots2,1,0)\]
Then place a bead on the abacus in each position given by $\beta(\lambda,b)$, so that there are a total of $b$ beads across the runners. Note that for a fixed value of $b$, the abacus is uniquely determined by $\lambda$, and any such abacus arrangement corresponds to a partition simply by reversing the above. Here is an example of such a construction:
	\begin{example}
		In this example we will fix the values $p=5,n=9,b=10$ and represent the partition $\lambda=(5,4)$ on the abacus. Following the above process, we first calculate the $\beta$-sequence of $\lambda$:
			\begin{eqnarray*}
				\beta(\lambda,10)	&=&(5-1+10,~4-2+10,\;-3+10,\;-4+10,\dots,\;-9+10,\;-10+10)\\
							&=&(14,12,7,6,5,4,3,2,1,0)
			\end{eqnarray*}
		The next step is to place beads on the abacus in the corresponding positions. We also number the beads, so that bead $1$ occupies position $\lambda_1-1+b$, bead $2$ occupies position $\lambda_2-2+b$ and so on. The labelled spaces and the final abacus are shown below.
			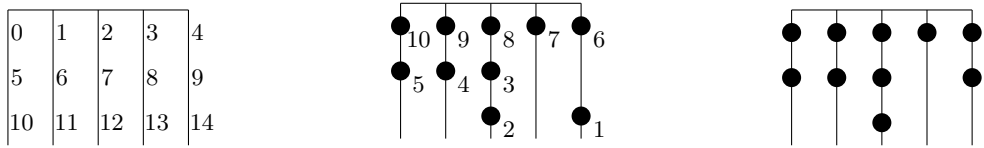
\begin{figure}[H]
				\centering
				\begin{tikzpicture}[scale=0.6]
					\draw (0,3)--(4,3);
					\foreach \x in {0,...,4}
						{\draw (\x,3)--(\x,0);
						\draw (\x+0.2,2.5) node {\small$\x$};}
					\foreach \x in {5,...,9}
						{\draw (\x-5+0.2,1.5) node {\small$\x$};}
					\foreach \x in {10,...,14}
						{\draw (\x-10+0.3,0.5) node {\small$\x$};}
				\end{tikzpicture}
				\hspace{2cm}
				\begin{tikzpicture}[scale=0.6]
					\draw (0,3)--(4,3);
					\foreach \x in {0,...,4}
						{\draw (\x,3)--(\x,0);}
					\foreach \x in {0,...,4}
						{\fill[black] (\x,2.5) circle (6pt);}
					\foreach \x in {0,1,2}
						{\fill[black] (\x,1.5) circle (6pt);}
					\foreach \x in {2,4}
						{\fill[black] (\x,0.5) circle (6pt);}
						
					\foreach \x in {10,...,6}
						{\draw (10-\x+0.4,2.5-0.3) node {\small$\x$};}
					\foreach \x in {5,4,3}
						{\draw (5-\x+0.4,1.5-0.3) node {\small$\x$};}
					\draw (2.4,0.2) node {\small$2$};
					\draw (4.4,0.2) node {\small$1$};
				\end{tikzpicture}
					\hspace{2cm}
				\begin{tikzpicture}[scale=0.6]
					\draw (0,3)--(4,3);
					\foreach \x in {0,...,4}
						{\draw (\x,3)--(\x,0);}
					\foreach \x in {0,...,4}
						{\fill[black] (\x,2.5) circle (6pt);}
					\foreach \x in {0,1,2,4}
						{\fill[black] (\x,1.5) circle (6pt);}
					\foreach \x in {2}
						{\fill[black] (\x,0.5) circle (6pt);}
				\end{tikzpicture}
				\caption{The positions on the abacus with 5 runners, the arrangement of beads (numbered) representing $\lambda=(5,4)$, and the corresponding 5-core.}\label{fig:symabacus}
			\end{figure}
	\end{example}
After fixing values of $p$ and $b$, we will abuse notation and write $\lambda$ for both the partition and the corresponding abacus with $p$ runners and $b$ beads. We then also define\linebreak \mbox{$\Gamma(\lambda,b)=(\Gamma(\lambda,b)_0,\Gamma(\lambda,b)_1,\dots,\Gamma(\lambda,b)_{p-1})$}, where 
	\begin{equation}\label{eq:gammaabacus}
		\Gamma(\lambda,b)_i=\big|\{j:\beta(\lambda,b)_j\equiv i\text{ (mod }p)\}\big|
	\end{equation}
so that $\Gamma(\lambda,b)$ records the number of beads on each runner of the abacus of $\lambda$.

We define the \emph{$p$-core} of a partition $\lambda$ to be the partition $\mu$ whose abacus is obtained from that of $\lambda$ by sliding all beads as far up their runners as possible. Note then that $\Gamma(\mu,b)=\Gamma(\lambda,b)$. It is shown in \cite[Chapter 2.7]{james1981representation} that this is independent of the choice of $b$. An example of this can be seen in Figure \ref{fig:symabacus}.

\subsection*{Specht Modules}

The algebra $R\sym_n$ is a cellular algebra, as shown in \cite{graham1996cellular}. The cell modules are labelled by the partitions of $n$, and are more commonly known as Specht modules. We denote the Specht module indexed by $\lambda$ by $S^\lambda_R$. These can be constructed explicitly, see for example \cite{james1978representation}. We then define the $K\sym_n$-module $S^\lambda_K=K\otimes_RS^\lambda_R$ and the $k\sym_n$-module $S^\lambda_k=k\otimes_RS^\lambda_R$.

	\begin{thm}[{\cite[Theorem 4.12]{james1978representation}}]
		The set $\{S^\lambda_K : \lambda\in\Lambda_n\}$ is a complete set of pairwise non-isomorphic simple $K\sym_n$-modules.
	\end{thm}

	\begin{thm}[{\cite[Theorem 11.5]{james1978representation}}]
		For $\lambda\in\Lambda^\ast_n$, the Specht module $S^\lambda_k$ has simple head $D^\lambda_k$. The set $\{D^\lambda_k : \lambda\in\Lambda^\ast_n\}$ is a complete set of pairwise non-isomorphic simple $k\sym_n$-modules.
	\end{thm}

	The blocks of the algebra $k\sym_n$ correspond to the $p$-cores of partitions in the following way.

	\begin{thm}[Nakayama's Conjecture, {\cite[Chapter 6]{james1981representation}}]
		Two partitions $\lambda,\mu\in\Lambda_n$ label Specht modules in the same block of $k\sym_n$ if and only if they have the same $p$-core, that is $\Gamma(\lambda,b)=\Gamma(\mu,b)$ for some (and hence all) $b\geq n$.
		\label{thm:nakayama}
	\end{thm}
	
\section{The partition algebra}

For a fixed $n\in\mathbb{N}$ and $\delta\in R$, we define the partition algebra $P_n^R(\delta)$ to be the set of linear combinations of set-partitions of $\{1,2,\dots,n,\bar1,\bar2,\dots,\bar n\}$. We call each part of a set-partition a \emph{block}. For instance,
	\[\big\{\{1,3,\bar3,\bar4\},\{2,\bar1\},\{4\},\{5,\bar2,\bar5\}\big\}\]
is a set-partition with $n=5$ consisting of 4 blocks. Any block with $\{i,\bar j\}$ as a subset for some $i$ and $j$ is called a \emph{propagating block}.

We can represent each set-partition by an \emph{$(n,n)$-partition diagram}, consisting of two rows of $n$ nodes with arcs between nodes in the same block. Multiplication in the partition algebra is by concatenation of diagrams in the following way: to obtain the result $x\cdot y$ given diagrams $x$ and $y$, place $x$ on top of $y$ and identify the bottom nodes of $x$ with those on top of $y$. This new diagram may contain a number, $t$ say, of blocks in the centre not connected to the northern or southern edges of the diagram. These we remove and multiply the final result by $\delta^t$. An example is given in Figure \ref{fig:partmult} below.

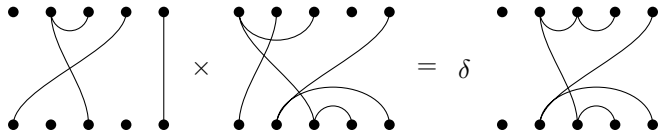
\begin{figure}[H]
	\centering
	\begin{tikzpicture}[scale=0.5]
	 \foreach \x in {1,...,5,7,8,...,11,14,15,...,18}
	 	{\fill[black] (\x,0) circle (4pt);
	 	\fill[black] (\x,3) circle (4pt);}
	 	
	\draw (3,3) arc (0:-180:0.5) .. controls (2,2) and (3,1) .. (3,0);
	\draw (4,3) .. controls (4,2) and (1,1)..(1,0);
	\draw (5,3)--(5,0);
	\draw (6,1.5) node {$\times$};
	\draw (9,3) arc (0:-180:1 and 0.75) ..controls (7,2) and (9,1).. (9,0) arc (180:0:0.5);
	\draw (8,3).. controls (8,2) and (7,1)..(7,0);
	\draw (11,0) arc (0:180:1.5 and 1) ..controls (8,1) and (11,2).. (11,3);
	\draw (12,1.5) node {$=$};\draw (13,1.5) node {$\delta$};
	\draw (18,0) arc (0:180:1.5 and 1) ..controls (15,1) and (18,2)..(18,3);
	\draw (17,3) arc (0:-180:0.5) arc (0:-180:0.5) ..controls (15,2) and (16,1)..(16,0) arc (180:0:0.5);
	\end{tikzpicture}
	\caption{Multiplication of two diagrams in $P_5^R(\delta)$.}\label{fig:partmult}
\end{figure}

As shown in Example \ref{ex:partdiagrams} below, there are many diagrams corresponding to the same set-partition. We will identify all such diagrams.

\begin{example}\label{ex:partdiagrams}
	Let $n=5$ and consider the set-partition $\big\{\{1,3,\bar3,\bar4\},\{2,\bar1\},\{4\},\{5,\bar2,\bar5\}\big\}$ as above. This can be represented by the diagrams in Figure \ref{fig:diagramreps}.
	\begin{figure}[H]
		\centering
		\begin{tikzpicture}[scale=0.5]
			\foreach \x in {1,...,5}
				{\fill[black] (\x,3) circle (4pt);
				\fill[black] (\x,0) circle (4pt);}
				\draw (1,3) arc (-180:0:1) -- (3,0) arc (180:0:0.5);
				\draw (2,3)..controls (2,2) and (1,1) ..(1,0);
				\draw (2,0) arc (180:0:1.5) -- (5,3);
		\end{tikzpicture}\hspace{1cm}
		\begin{tikzpicture}[scale=0.5]
			\foreach \x in {1,...,5}
				{\fill[black] (\x,3) circle (4pt);
				\fill[black] (\x,0) circle (4pt);}
				\draw (1,3) arc (-180:0:1) ..controls (3,2) and (4,1).. (4,0) arc (0:180:0.5);
				\draw (2,3)..controls (2,2) and (1,1) ..(1,0);
				\draw (2,0) arc (180:0:1.5 and 1) -- (5,3) ..controls (5,2) and (2,1).. (2,0);
		\end{tikzpicture}\hspace{1cm}
		\begin{tikzpicture}[scale=0.5]
			\foreach \x in {1,...,5}
				{\fill[black] (\x,3) circle (4pt);
				\fill[black] (\x,0) circle (4pt);}
				\draw (3,3) arc (0:-180:1) .. controls (1,2) and (3,1).. (3,0) arc (180:0:0.5);
				\draw (2,3)..controls (2,2) and (1,1) ..(1,0);
				\draw (5,0) arc (0:180:1.5 and 1) .. controls (2,1) and (5,2) .. (5,3);
		\end{tikzpicture}
		\caption{Three ways of representing the given set-partition}\label{fig:diagramreps}
	\end{figure}
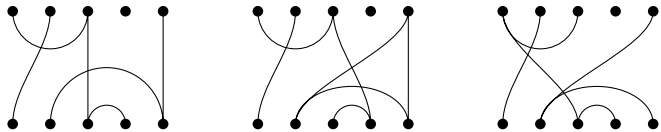
\end{example}

The following elements of $P_n^R(\delta)$ will be of interest:

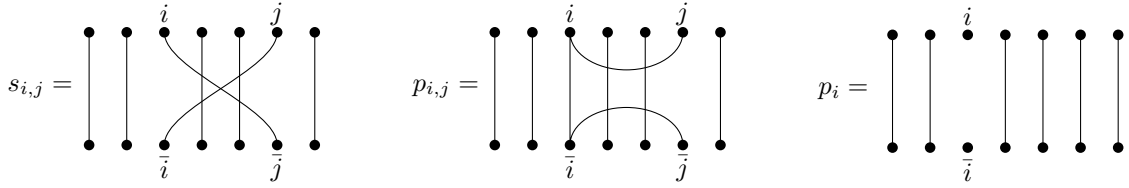
\begin{figure}[H]
\centering
\begin{tikzpicture}[scale=0.5]
	\draw (-0.3,1.5) node {$s_{i,j}=$};
	\foreach \x in {1,...,7}
	{\fill[black] (\x,0) circle (4pt);
	\fill[black] (\x,3) circle (4pt);}
	\foreach \x in {1,2,4,5,7}
	{\draw (\x,3)--(\x,0);}
	\draw (3,3) .. controls (3,2) and (6,1) .. (6,0);
	\draw (6,3) .. controls (6,2) and (3,1) .. (3,0);
	\draw (3,3.5) node {$i$};\draw (6,3.5) node {$j$};
	\draw (3,-0.6) node {$\bar i$};\draw (6,-0.6) node {$\bar j$};
\end{tikzpicture}\hspace{1cm}
\begin{tikzpicture}[scale=0.5]
	\draw (-0.3,1.5) node {$p_{i,j}=$};
	\foreach \x in {1,...,7}
	{\fill[black] (\x,0) circle (4pt);
	\fill[black] (\x,3) circle (4pt);}
	\foreach \x in {1,2,4,5,7}
	{\draw (\x,3)--(\x,0);}
	\draw (6,3) arc (0:-180:1.5 and 1) -- (3,0) arc (180:0:1.5 and 1);
	\draw (3,3.5) node {$i$};\draw (6,3.5) node {$j$};
	\draw (3,-0.6) node {$\bar i$};\draw (6,-0.6) node {$\bar j$};
\end{tikzpicture}\hspace{1cm}
\begin{tikzpicture}[scale=0.5]
	\draw (-0.3,1.5) node {$p_{i}=$};
	\foreach \x in {1,...,7}
	{\fill[black] (\x,0) circle (4pt);
	\fill[black] (\x,3) circle (4pt);}
	\foreach \x in {1,2,4,5,6,7}
	{\draw (\x,3)--(\x,0);}
	\draw (3,3.5) node {$i$};
	\draw (3,-0.6) node {$\bar i$};
\end{tikzpicture}

\end{figure}

It was shown in \cite{halverson2005partition} that these elements generate $P^R_n(\delta)$.

Notice that multiplication in $P_n^R(\delta)$ cannot increase the number of propagating blocks. We therefore have a filtration of $P_n^R(\delta)$ by the number of propagating blocks. Over $\bbf$, we can construct this filtration explicitly by use of the idempotents $e_i$ defined in Figure \ref{fig:partidempotent} below.
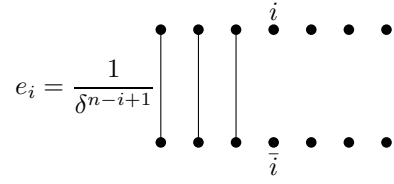
\begin{figure}[H]
	\centering
	\begin{tikzpicture}[scale=0.5]
		\draw (-1,1.5) node {$\displaystyle e_i=\frac{1}{\delta^{n-i+1}}$};
		\foreach \x in {1,...,7}
		{\fill[black] (\x,0) circle (4pt);
		\fill[black] (\x,3) circle (4pt);}
		\foreach \x in {1,2,3}
		{\draw (\x,0)--(\x,3);}
		\draw (4,3.5) node  {$i$};\draw (4,-0.6) node {$\bar i$};
	\end{tikzpicture}\caption{The idempotent $e_i$}\label{fig:partidempotent}
\end{figure}
The filtration is then given by
	\begin{equation}\label{eq:partfiltration}
		J_n^{(0)}\subset J_n^{(1)}\subset \dots J_n^{(n-1)}\subset J_n^{(n)}=P_n^\bbf(\delta)
	\end{equation}
	where $J_n^{(r)}=P_n^\bbf(\delta)e_{r+1}P_n^\bbf(\delta)$ contains only diagrams with at most $r$ propagating blocks. We also use $e_i$ to construct algebra isomorphisms
	\begin{equation}\label{eq:partiso}
		\Phi_n:P_{n-1}^\bbf(\delta)\longrightarrow e_nP_n^\bbf(\delta)e_n
	\end{equation}
	taking a diagram in $P_{n-1}^\bbf(\delta)$ and adding an extra northern and southern node to the right hand end.
Using this and following \cite{green1980polynomial} we obtain an exact localisation functor
	\begin{eqnarray}
		F_n:P_n^\bbf(\delta)\text{-\catmod}	&\longrightarrow&	P_{n-1}^\bbf(\delta)\text{-\catmod}\label{eq:partfn}\\
					 			 M 	&\longmapsto&    	e_nM\nonumber
	\end{eqnarray}
and a right exact globalisation functor
	\begin{eqnarray}
		G_n:P_n^\bbf(\delta)\text{-\catmod}	&\longrightarrow&	P_{n+1}^\bbf(\delta)\text{-\catmod}\label{eq:partgn}\\
					  M 				&\longmapsto&	P_{n+1}^\bbf(\delta)e_{n+1}\otimes_{P_n^\bbf(\delta)}M\nonumber
	\end{eqnarray}
Since $F_{n+1}G_n(M)\cong M$ for all $M\in P_n^\bbf(\delta)$-\catmod, $G_n$ is a full embedding of categories. From the filtration \eqref{eq:partfiltration} we see that
	\begin{equation}\label{eq:partlift}
		P_n^\bbf(\delta)/J_n^{(n-1)}\cong\bbf \sym_n
	\end{equation}
	and so using \eqref{eq:partiso} and following \cite{green1980polynomial}, we see that the simple $P_n^\bbf(\delta)$-modules are indexed by the set $\Lambda_{\leq n}$ if $\bbf=K$ and by the set $\Lambda_{\leq n}^\ast$ if $\bbf=k$ (see \eqref{eq:partsets} and \eqref{eq:regpartsets} for definitions of these).
	
We will also need to consider the algebra $P_{n-\frac{1}{2}}^\bbf(\delta)$, which is the subalgebra of $P_n^\bbf(\delta)$ spanned by all set-partitions with $n$ and $\bar n$ in the same block. As in \eqref{eq:partfiltration} we have a filtration of this algebra defined by the number of propagating blocks: 
	\begin{equation}\label{eq:partfiltrationhalf}
		J_{n-\frac{1}{2}}^{(1)}\subset J_{n-\frac{1}{2}}^{(2)}\subset \dots J_{n-\frac{1}{2}}^{(n-1)}\subset J_{n-\frac{1}{2}}^{(n)}=P_{n-\frac{1}{2}}^\bbf(\delta)
	\end{equation}
where $J_{n-\frac{1}{2}}^{(r)}$ contains all diagrams with at most $r$ propagating blocks. Note that since we require the nodes $n$ and $\bar n$ to be in the same block, we always have at least one propagating block. Also since $n$ and $\bar n$ must always be joined, we see that $P_{n-\frac{1}{2}}^\bbf(\delta)/J_{n-\frac{1}{2}}^{(n-1)}\cong\bbf\sym_{n-1}$, and so following the argument for $P_n^\bbf(\delta)$ above we see that the simple $P_{n-\frac{1}{2}}^\bbf(\delta)$-modules are indexed by $\Lambda_{\leq n-1}$ if $\bbf=K$ and by $\Lambda_{\leq n-1}^\ast$ if $\bbf=k$.
\\~\\
Note that we have a natural inclusion of $P_n^R(\delta)$ inside $P_{n+\half}^R(\delta)$ 	
	\begin{align*}
		P_n^R(\delta)&\longrightarrow P_{n+\frac{1}{2}}^R(\delta)\\
		d&\longmapsto d\cup \big\{\{n+1,\overline{n+1}\}\big\}
	\end{align*}
This allows us to define restriction and induction functors
\begin{align}
\mathrm{res}_n:P_n^\bbf(\delta)\text{-\catmod}&\longrightarrow P^\bbf_{n-\frac{1}{2}}(\delta)\text{-\catmod}\nonumber\\
				M&\longmapsto M|_{P_{n-\half}^\bbf(\delta)}\nonumber\\\nonumber\\
\mathrm{ind}_n:P^\bbf_n(\delta)\text{-\catmod}&\longrightarrow P^\bbf_{n+\frac{1}{2}}(\delta)\text{-\catmod}\nonumber\\
				M&\longmapsto P_{n+\half}^\bbf(\delta)\otimes_{P_n^\bbf(\delta)}M\label{eq:partresind}
\end{align}

\subsection*{Cellularity of $P_n^\bbf(\delta)$}\label{sec:partcellularity}

It was shown in \cite{xi1999partition} that the partition algebra is cellular. The cell modules $\Delta_\lambda^\bbf(n;\delta)$ are indexed by partitions $\lambda\in\Lambda_{\leq n}$, and the cellular ordering is given by the reverse of $<_d$. When $\lambda\vdash n$, we obtain $\Delta_\lambda^\bbf(n;\delta)$ by lifting the Specht module $S^\lambda_\bbf$ to the partition algebra using \eqref{eq:partlift}. When $\lambda\vdash n-t$ for some $t>0$, we obtain the cell module by
	\[\Delta_\lambda^\bbf(n;\delta)=G_{n-1}G_{n-2}\dots G_{n-t}\Delta_\lambda^\bbf(n-t;\delta)\]
	
Over $K$, each of the cell modules has a simple head $L_\lambda^K(n;\delta)$, and these form a complete set of non-isomorphic simple $P_n^K(\delta)$-modules. Over $k$, the heads $L_\lambda^k(n;\delta)$ of cell modules labelled by $p$-regular partitions $\lambda\in\Lambda_{\leq n}^\ast$ provide a complete set of non-isomorphic simple $P_n^k(\delta)$-modules.

When the context is clear, we will write $\Delta_\lambda^\bbf(n)$ and $L_\lambda^\bbf(n)$ to mean $\Delta_\lambda^\bbf(n;\delta)$ and $L_\lambda^\bbf(n;\delta)$ respectively.

We also have an explicit construction of the cell modules. Let $I(n,t)$ be the set of $(n,n)$-diagrams with precisely $t$ propagating blocks and $\overline{t+1},\overline{t+2},\dots,\overline{n}$ each in singleton blocks. Then denote by $V(n,t)$ the free $R$-module with basis $I(n,t)$. There is a $(P^R_n(\delta),\sym_t)$-bimodule action on $V(n,t)$, where elements of $P_n^R(\delta)$ act on the left by concatenation as normal and elements of $\sym_t$ act on the right by permuting the $t$ leftmost southern nodes.
Thus for a partition $\lambda\vdash t$ we can easily show that $\Delta_\lambda^R(n)\cong V(n,t)\otimes_{\sym_t}S_R^\lambda$, where $S_R^\lambda$ is the Specht module. The action of $P_n^R(\delta)$ on $\Delta_\lambda^R(n)$ is as follows: given a partition diagram $x\in P_n^R(\delta)$ and a pure tensor $v\otimes s\in\Delta_\lambda^R(n)$, we define the element
	\[x(v\otimes s)=(xv)\otimes s\]
where $(xv)$ is the product of two diagrams in the usual way if the result has $t$ propagating blocks, and is 0 otherwise.
	\begin{rem*}
		Note that we \textbf{cannot} in general provide an $R$-module $L^R_\lambda(n)$ such that $L_\lambda^K(n)=K\otimes_R L_\lambda^R(n)$ or $L_\lambda^k(n)=k\otimes_R L_\lambda^R(n)$.
	\end{rem*}

The algebra $P_{n-\frac{1}{2}}^\bbf(\delta)$ is also cellular (see \cite{martin2000partition}). We can construct the cell modules in a similar way to those of $P_n^\bbf(\delta)$. Let $I(n-\frac{1}{2},t)$ be the set of $(n,n)$-diagrams with precisely $t$ propagating blocks, one of which contains $n$ and $\bar n$, with $\overline{t+1},\overline{t+2},\dots,\overline{n-1}$ each in singleton blocks. Then denote by $V(n-\frac{1}{2},t)$ the free $R$-module with basis $I(n-\frac{1}{2},t)$. There is a $(P^R_{n-\frac{1}{2}}(\delta),\sym_{t-1})$-bimodule action on $V(n,t)$, where elements of $P_{n-\frac{1}{2}}^R(\delta)$ act on the left as normal and elements of $\sym_{t-1}$ act on the right by permuting the $t-1$ leftmost southern nodes.
Then for a partition $\lambda\vdash t-1$ we have $\Delta_\lambda^R(n-\frac{1}{2})\cong V(n-\frac{1}{2},t)\otimes_{\sym_{t-1}}S_R^\lambda$, where $S_R^\lambda$ is a Specht module. Note that when $\lambda\vdash n-1$, $\Delta_\lambda^R(n-\frac{1}{2})\cong S^\lambda_R$, the Specht module. The action of $P_{n-\frac{1}{2}}^R(\delta)$ is the same as in the previous case.

We then have
	\[\Delta_\lambda^K(n-\textstyle\half)=K\otimes_R\Delta_\lambda^R(n-\half)~~~~~ \text{and}~~~~~\Delta_\lambda^k(n-\half)=k\otimes_R\Delta_\lambda^R(n-\half)\]
and as before, $\Delta_\lambda^K(n-\half)$ has a simple head $L_\lambda^K(n-\half)$ for all $\lambda$, and $\Delta_\lambda^k(n-\half)$ has a simple head $L_\lambda^k(n-\half)$ for each $p$-regular $\lambda$.
\\~\\
The localisation and globalisation functors (\eqref{eq:partfn} and \eqref{eq:partgn}) preserve the cellular structure of the partition algebra, and in particular map cell modules to cell modules as below:
\begin{align*}
	F_n(\Delta_\lambda^\bbf(n))&\cong\begin{cases}
									\Delta_\lambda^\bbf(n-1)&\text{ if }\lambda\in\Lambda_{\leq n-1}\\
									0&\text{ otherwise}
									\end{cases}\\
	G_n(\Delta_\lambda^\bbf(n))&\cong\Delta_\lambda^\bbf(n+1)
\end{align*}		
It was shown in \cite[Proposition 7]{martin2000partition} that the restriction and induction functors \eqref{eq:partresind} also preserve the cellular structure of $P_n^\bbf(\delta)$. Furthermore, if we apply these to cell modules, then the result has a filtration by cell modules. In particular, we have the following exact sequences:
\begin{align}0\longrightarrow\Delta_\lambda^\bbf(n)\longrightarrow \mathrm{res}_{n+{\half}}\Delta_\lambda^\bbf(n+\textstyle{\half}) \longrightarrow\displaystyle\biguplus_{\mu\triangleright\lambda}\Delta_\mu^\bbf(n) \longrightarrow0\nonumber\\[10pt]
0\longrightarrow\biguplus_{\mu\triangleleft\lambda}\Delta_\mu^\bbf(n-\textstyle{\half})\longrightarrow \mathrm{res}_{n}\Delta_\lambda^\bbf(n) \longrightarrow\Delta_\lambda^\bbf(n-\half) \longrightarrow0\nonumber\\[15pt]
0\longrightarrow\Delta_\lambda^\bbf(n)\longrightarrow \mathrm{ind}_{n-{\half}}\Delta_\lambda^\bbf(n-\textstyle\half)\longrightarrow\displaystyle\biguplus_{\mu\triangleright\lambda}\Delta_\mu^\bbf(n) \longrightarrow0\nonumber\\[10pt]
0\longrightarrow\biguplus_{\mu\triangleleft\lambda}\Delta_\mu^\bbf(n+\textstyle{\half})\longrightarrow \mathrm{ind}_n\Delta_\lambda^\bbf(n)\longrightarrow\Delta_\lambda^\bbf(n+\half) \longrightarrow0\label{eq:partexact}\end{align}
The following result from Martin \cite[Section 3]{martin2000partition} allows us to focus on the partition algebras $P_n^\bbf(\delta)$ with $n\in\mathbb{Z}$.
\begin{prop}\label{prop:morita}
	Define the idempotent
		\[\xi_{n+1}=\prod_{i=1}^n(1-p_{i,n+1})\in P_{n+1}^\bbf(\delta)\]
	Then we have an algebra isomorphism
		\[\xi_{n+1}P_{n+\frac{1}{2}}^\bbf(\delta)\xi_{n+1}\cong P_n^\bbf(\delta-1)\]
	which induces a Morita equivalence between the categories $P_{n+\frac{1}{2}}^\bbf(\delta)\text{-\catmod}$ and $P_n^\bbf(\delta-1)\text{-\catmod}$. More precisely, using the above isomorphism the functors
	\begin{align*}
		\Phi:P_{n+\frac{1}{2}}^\bbf(\delta)\text{-\catmod}&\longrightarrow P_n^\bbf(\delta-1)\text{-\catmod}\\
												M&\longmapsto\xi_{n+1}M\\
		\text{and~~~~~}\Psi:P_n^\bbf(\delta-1)\text{-\catmod}&\longrightarrow \Phi:P_{n+\frac{1}{2}}^\bbf(\delta)\text{-\catmod}\\
												N&\longmapsto P_{n+\textstyle\half}^\bbf(\delta)\xi_{n+1}\otimes_{P_n^\bbf(\delta-1)}N
	\end{align*}
	define an equivalence of categories. Moreover, this equivalence preserves the cellular structure of these algebras and we have
		\[\Phi(\Delta_\lambda^\bbf(n+\textstyle\half))\cong\Delta_\lambda^\bbf(n)\]
	for all $\lambda\in\Lambda_{P_n}$.
\end{prop}

\subsection*{The blocks of the partition algebra}

The blocks of the partition algebra $P^K_n(\delta)$ in characteristic 0 were described in \cite{martin1996structure}. Assuming $\delta$ is an integer (otherwise the algebra is semisimple), the blocks are given by chains of partitions, each satisfying a combinatorial property determined by the previous partition in the chain.  We briefly recount this below, but first we introduce some notation.

\begin{defn}\label{def:char0blocks}
	Let $\mathcal{B}^K_\lambda(n;\delta)$ be the set of partitions $\mu$ labelling cell modules in the same block as $\Delta^K_\lambda(n)$. We will also say that partitions $\mu$ and $\lambda$ lie in the same block if they label cell modules in the same block. If the context is clear, we will write $\mathcal{B}^K_\lambda(n)$ to mean $\mathcal{B}_\lambda^K(n;\delta)$.
\end{defn}

\begin{defn}\label{def:deltapair}
	Let $\lambda,\mu$ be partitions, with $\mu\subset\lambda$. We say that $(\mu,\lambda)$ is a \emph{$\delta$-pair}, written $\mu\hookrightarrow_\delta\lambda$, if $\lambda$ differs from $\mu$ by a strip of nodes in a single row, the last of which has content $\delta-|\mu|$.
\end{defn}
The following is an example of this condition.
\begin{example}
We let $\delta=7$, $\lambda=(4,3,1)$ and $\mu=(4,1,1)$. Then we see that $\lambda$ and $\mu$ differ in precisely one row, and the last node in this row of $\lambda$ has content $1$ (see Figure \ref{fig:deltapair}). Since $\delta-|\mu|=7-6=1$, we see that $(\mu,\lambda)$ is a $7$-pair.

\begin{figure}[H]
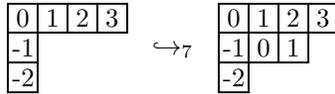

	\centering
	\newcommand{\mone}{\text{-1}}
	\newcommand{\mtwo}{\text{-2}}
	$\Yvcentermath1\young(0123,\mone,\mtwo)~~\hookrightarrow_7~~\young(0123,\mone01,\mtwo)$
	\caption{An example of a $\delta$-pair when $\delta=7$}\label{fig:deltapair}
\end{figure}
\end{example}
We then have the following characterisation of the blocks of the partition algebra in characteristic 0.
\begin{thm}[{\cite[Proposition 9]{martin1996structure}}]\label{thm:martinblocks}
Each block of the partition algebra $P^K_n(\delta)$ is given by a chain of partitions
	\[\lambda^{(0)}\subset\lambda^{(1)}\subset\dots\subset\lambda^{(r)}\]
where for each $i$, ($\lambda^{(i)}$,$\lambda^{(i+1)}$) form a $\delta$-pair, differing in the $(i+1)$-th row. Moreover there is an exact sequence of $P_n^K(\delta)$-modules
	\[0\rightarrow\Delta_{\lambda^{(r)}}^K(n)\rightarrow\Delta_{\lambda^{(r-1)}}^K(n)\rightarrow \dots\rightarrow\Delta_{\lambda^{(1)}}^K(n)\rightarrow\Delta_{\lambda^{(0)}}^K(n)\rightarrow L_{\lambda^{(0)}}^K(n)\rightarrow0\]
	with the image of each homomorphism a simple module. In particular, each of the cell modules $\Delta_{\lambda^{(i)}}^K(n)$ for $0\leq i<r$ has Loewy structure
	\begin{center}
		$L_{\lambda^{(i)}}^K(n)$\\
		$L_{\lambda^{(i+1)}}^K(n)$
	\end{center}
	and $\Delta_{\lambda^{(r)}}^K(n)=L_{\lambda^{(r)}}^K(n)$.
\end{thm}

This was reformulated in $\cite{bowmanblocks}$ as a geometric characterisation in the following way. 

Let $\{\varepsilon_0,\dots,\varepsilon_n\}$ be a set of formal symbols and set
	\[E_n=\bigoplus_{i=0}^n\mathbb{R}\varepsilon_i\]
We have an inner product $\langle~,~\rangle$ on $E_n$ given by extending linearly the relations
	\[\langle\varepsilon_i,\varepsilon_j\rangle=\delta_{ij}\]
where $\delta_{ij}$ is the Kronecker delta.\\
Let $\Phi_n=\{\varepsilon_i-\varepsilon_j:0\leq i,j\leq n\}$ be a root system of type $A_n$, and $W_n\cong\sym_{n+1}$ the corresponding Weyl group, generated by the reflections $s_{i,j}=s_{\varepsilon_i-\varepsilon_j}$ ($0\leq i<j\leq n$). There is an action of $W_n$ on $E_n$, the generators acting by
	\[s_{i,j}(x)=x-\langle x,\varepsilon_i-\varepsilon_j\rangle(\varepsilon_i-\varepsilon_j)\]
	for all $x\in E_n$.
	
	If we fix the element $\rho=\rho(\delta)=(\delta,-1,-2,\dots,-n)$ we may then define a shifted action of $W_n$ on $E_n$, given by
	\[w\cdot_\delta x=w(x+\rho(\delta))-\rho(\delta)\]
for all $w\in W_n$ and $x\in E_n$.

	Given a partition $\lambda=(\lambda_1,\lambda_2,\dots,\lambda_l)$, let 
\[\hat\lambda=(-|\lambda|,\lambda_1,\dots,\lambda_n)=-|\lambda|\varepsilon_0+\sum_{i=1}^n\lambda_i\varepsilon_i\in E_n\]
where any $\lambda_i$ not appearing in $\lambda$ is taken to be zero. Using this embedding of $\Lambda_{\leq n}$ into $E_n$ we can consider the action of $W_n$ on the set of partitions $\Lambda_{\leq n}$ defined by
\[w\cdot_\delta\hat\lambda=w(\hat\lambda+\rho(\delta))-\rho(\delta)\]
where $w\in W_n$ and $\rho(\delta)=(\delta,-1,-2,\dots,-n)$ as before. We introduce the following notation for the orbits of this action

\begin{defn}\label{def:char0orbits}
	Let $\mathcal{O}_\lambda(n;\delta)$ be the set of partitions $\mu$ such that $\hat\mu\in W_n\cdot_\delta\hat\lambda$. If the context is clear, we will write $\mathcal{O}_\lambda(n)$ to mean $\mathcal{O}_\lambda(n;\delta)$.
\end{defn}

  We then have the following reformulation of \cite{martin1996structure}.
\begin{thm}[{\cite{martin1996structure}, \cite[Theorem 4.5]{bowmanblocks}}]\label{thm:part0blocks}
	For all $\lambda\in\Lambda_{\leq n}$, we have $\mathcal{B}_\lambda^K(n;\delta)=\mathcal{O}_\lambda(n;\delta)$.
\end{thm}

This can be extended to give a characterisation of the blocks of the the partition algebra $P_n^k(\delta)$ in characteristic $p$. We let $W_n^p$ be the affine Weyl group corresponding to $\Phi_n$, generated by reflections $s_{i,j,rp}=s_{\varepsilon_i-\varepsilon_j,rp}$ $(0\leq i<j\leq n,r\in\mathbb{Z})$, with an action on $E_n$ given by
\[s_{i,j,rp}(x)=x-\big(\langle x,\varepsilon_i-\varepsilon_j\rangle-rp\big)(\varepsilon_i-\varepsilon_j)\]
We use the following notation, analogues of Definitions \ref{def:char0blocks} and \ref{def:char0orbits}.

\begin{defn}\label{def:charporbits}
	Let $\mathcal{B}^k_\lambda(n;\delta)$ be the set of partitions $\mu$ labelling cell modules in the same block as $\Delta^k_\lambda(n;\delta)$. We will also say that partitions $\mu$ and $\lambda$ lie in the same block if they label cell modules in the same block. Moreover, let $\mathcal{O}^p_\lambda(n;\delta)$ be the set of partitions $\mu$ such that $\hat\mu\in W^p_n\cdot_\delta\hat\lambda$. If the context is clear, we will write $\mathcal{B}_\lambda^k(n)$ and $\mathcal{O}^p_\lambda(n)$ to mean $\mathcal{B}_\lambda^k(n;\delta)$ and $\mathcal{O}^p_\lambda(n;\delta)$ respectively.
\end{defn}

\begin{thm}[{\cite[Theorem 5.19]{bowmanblocks}}]~\label{thm:pblocks}
For all $\lambda\in\Lambda_{\leq n}$, we have $\mathcal{B}_\lambda^k(n;\delta)=\mathcal{O}^p_\lambda(n;\delta)$\end{thm}

The proof of Theorem \ref{thm:pblocks} given in \cite{bowmanblocks} introduces a varation of the abacus as defined in Section \ref{sec:symmetricgroup}. We will briefly outline this below.
\\~\\
For any two partitions $\lambda,\mu\in\Lambda_{\leq n}$ it is possible to show that
	\begin{equation}\label{eqn:orbitssequence}
		\mu\in\mathcal{O}^p_\lambda(n;\delta)\iff\hat\mu+\rho(\delta)\sim_p\hat\lambda+\rho(\delta)
	\end{equation}
where $\sim_p$ means that the two sequences modulo $p$ are the same up to reordering.

We represent this equivalence in the form of an abacus in the following way. For a partition $\lambda$, choose $b\in\mathbb{N}$ satisfying $b\geq|\lambda|$. We write $\hat\lambda$ as a $(b+1)$-tuple by adding zeros to obtain a vector in $E_b$, and extend $\rho(\delta)$ to the $(b+1)$-tuple
	\[\rho(\delta)=(\delta,-1,-2,\dots,-b)\in E_b\]
Now define the $\beta_\delta$-sequence of $\lambda$ to be
	\begin{align*}
		\beta_\delta(\lambda,b)&=\hat\lambda+\rho(\delta)+b(\underbrace{1,1,\dots,1}_{n+1})\\
						&=(\delta-|\lambda|+b,\lambda_1-1+b,\lambda_2-2+b,\dots,\lambda_l-l+b,-(l+1)+b,\dots,2,1,0)
	\end{align*}
We then see that \eqref{eqn:orbitssequence} is also equivalent to $\beta_\delta(\mu,b)\sim_p\beta_\delta(\lambda,b)$. The $\beta_\delta$-sequence is used to construct the \emph{$\delta$-marked abacus} of $\lambda$ as follows:
	\begin{enumerate}
		\item Take an abacus with $p$ runners, labelled $0$ to $p-1$ from left to right. The positions of the abacus start at 0 and increase from left to right, moving down the runners.
		\item Let $\beta_\delta(\lambda,b)_0=\delta-|\lambda|+b\equiv v_\lambda$ (mod $p$), where $0\leq v_\lambda\leq p-1$. Place a $\vee$ on top of runner $v_\lambda$.
		\item For the rest of the entries of $\beta_\delta(\lambda,b)$, place a bead in the corresponding position of the abacus, so that the final abacus contains $b$ beads. 
	\end{enumerate}
Example \ref{ex:abacus} below demonstrates this construction.
\begin{example}\label{ex:abacus}
Let $p=5$, $\delta=6$, $\lambda=(2,1)$. We choose an integer $b\geq3$, for instance $b=7$. Then the $\beta$-sequence is
	\begin{align*}
		\beta_\delta(\lambda,7)&=(6-3+7,2-1+7,\dots,0)\\
						&=(10,8,6,4,3,2,1,0)		
	\end{align*}
	The resulting abacus is given in Figure \ref{fig:abacus}.
	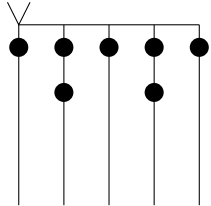
\begin{figure}[H]
		\centering
		\begin{tikzpicture}[scale=0.6]
			\draw (1,4)--(5,4);
			\foreach \x in {1,...,5}
			{\draw (\x,0)--(\x,4);
			\fill[black] (\x,3.5) circle (6pt);}
			\fill[black] (2,2.5) circle (6pt);
			\fill[black] (4,2.5) circle (6pt);
			\draw (0.75,4.5)--(1,4)--(1.25,4.5);
		\end{tikzpicture}
		\caption{The $\delta$-marked abacus of $\lambda$, where $\lambda=(2,1)$, $p=5$, $\delta=1$ and $b=7$.}\label{fig:abacus}
	\end{figure}
\end{example} 
	Note that if we ignore the $\vee$ we recover James' abacus representing $\lambda$ with $b$ beads.
	
	If the context is clear, we will use \emph{marked abacus} to mean $\delta$-marked abacus.
	
Recall the definition of $\Gamma(\lambda,b)$ from Section \ref{sec:symmetricgroup}. If we now use the marked abacus, we similarly define $\Gamma_\delta(\lambda,b)=(\Gamma_\delta(\lambda,b)_0,\Gamma_\delta(\lambda,b)_1,\dots,\Gamma_\delta(\lambda,b)_{p-1})$ by
	\[\Gamma_\delta(\lambda,b)_i=\begin{cases}
								\Gamma(\lambda,b)&\text{if }i\neq v_\lambda\\
								\Gamma(\lambda,b)+1&\text{if }i=v_\lambda
							\end{cases}\]
Given any other partition $\mu$, we construct its marked abacus and see that a further equivalent form of \eqref{eqn:orbitssequence} is $\Gamma_\delta(\mu,b)=\Gamma_\delta(\lambda,b)$. Combining this with Theorem \ref{thm:pblocks} gives a characterisation of the blocks of $P_n^k(\delta)$ in terms of the marked abacus.
	
\section{The decomposition matrix of $P_n^k(\delta)$}

In this section we present some results that allow us to use information about $P_r^K(\delta+tp)$ $(t\in\mathbb{Z})$ to understand the structure of $P_r^k(\delta)$. We use the notation $\mathbf{D}(A)$ to denote the decomposition matrix of the algebra $A$.

We first recall the following theorem from \cite{hartmann2010cohomological} which allows us to use the modular representation theory of the symmetric group in examining the partition algebra.

\begin{thm}[{\cite[Corollary 6.2]{hartmann2010cohomological}}]
		Let $\lambda,\mu\vdash n-t$ be partitions, with $\lambda\in\Lambda_{\leq n}^\ast$. Then 
			\[[\Delta_\mu^k(n;\delta):L_\lambda^k(n;\delta)]=[S^\mu_k:D^\lambda_k]\]
		In particular, given two partitions $\lambda,\mu\vdash n-t$, if the two Specht modules $S^\lambda_k$ and $S^\mu_k$ are in the same block for the symmetric group algebra $k\sym_{n-t}$, then $\mu\in\mathcal{B}_\lambda^k(n;\delta)$.\label{thm:partsymblocks}
\end{thm}

We also recall some results from \cite{doran2000partition} which can be generalised to fields of arbitrary characteristic. We begin by defining the $k$-vector space $\Psi(n,t)=\{u\in k\otimes_RV(n,t):p_{i,j}u=0\text{ for all }i\neq j\}$.

\begin{defn}
	We place a partial order $\prec$ on $I(n,t)$ by refinement of set-partitions. Let $M(n,t)$ be the set of minimal elements of $I(n,t)$ under $\prec$.
	
	For $x,y\in I(n,t)$, we recursively define the M\"{o}bius function to be
	\[\mu(x,y)=\begin{cases}
					1&\text{if }x=y\\
					-\sum_{x\preceq z\prec y}\mu(x,z)&\text{if }x\prec y\\
					0&\text{otherwise}
				\end{cases}\]
\end{defn}

\begin{example}
	The Hasse diagram of $I(3,1)$ under $\prec$ is given below.
		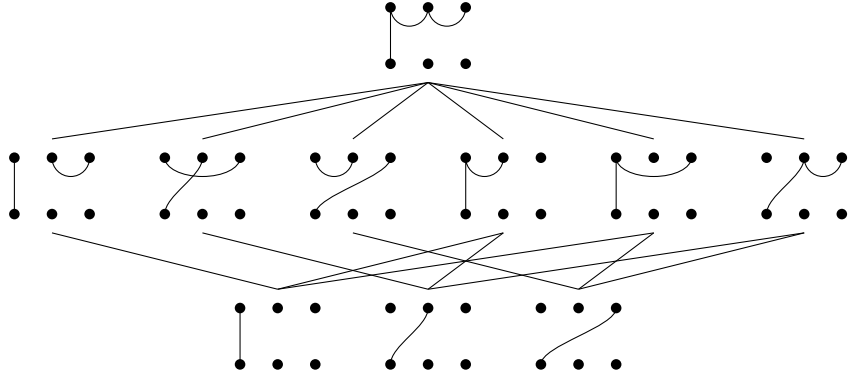
\begin{figure}[H]
			\centering
			\begin{tikzpicture}[scale=0.5]
				\foreach \x in {7,8,9,11,12,13,15,16,17}
					{\fill[black] (\x,0) circle (4pt);
					\fill[black] (\x,1.5) circle (4pt);}
					\draw (7,0)--+(0,1.5);
					\draw (11,0).. controls (11,0.5) and (12,1) .. (12,1.5);
					\draw (15,0).. controls (15,0.5) and (17,1) .. (17,1.5);
					
				\draw (8,2)--(2,3.5);\draw (8,2)--(14,3.5);\draw (8,2)--(18,3.5);
				\draw (12,2)--(6,3.5);\draw (12,2)--(14,3.5);\draw (12,2)--(22,3.5);
				\draw (16,2)--(10,3.5);\draw (16,2)--(18,3.5);\draw (16,2)--(22,3.5);
					
				\foreach \x in {1,2,3,5,6,7,9,10,11,13,14,15,17,18,19,21,22,23}
					{\fill[black] (\x,4) circle (4pt);
					\fill[black] (\x,5.5) circle (4pt);}
					\draw (1,4)--+(0,1.5);\draw (2,5.5) arc (-180:0:0.5);
					\draw (5,4)..controls (5,4.5) and (6,5) .. (6,5.5);\draw (5,5.5) arc (-180:0:1 and 0.5);
					\draw (9,4)..controls (9,4.5) and (11,5) ..(11,5.5);\draw (9,5.5) arc (-180:0:0.5);
					\draw (13,4)--++(0,1.5) arc (-180:0:0.5);
					\draw (17,4)--++(0,1.5) arc (-180:0:1 and 0.5);
					\draw (21,4)..controls (21,4.5) and (22,5) .. (22,5.5) arc (-180:0:0.5);
					
				\foreach \x in {0,1,2,3,4,5}
				\draw (4*\x+2,6)--(12,7.5);

				\foreach \x in {11,12,13}
					{\fill[black] (\x,8) circle (4pt);
					\fill[black] (\x,9.5) circle (4pt);}
				\draw (11,8)--++(0,1.5)arc(-180:0:0.5)arc(-180:0:0.5);
			\end{tikzpicture}
			\caption{The Hasse diagram of $I(3,1)$.}
		\end{figure}
	The three diagrams on the bottom row are the elements of $M(3,1)$.
\end{example}
The proof of the following proposition is valid over a field of positive characteristic.
\begin{prop}[{\cite[Proposition 4.3]{doran2000partition}}]
	A basis for $\Psi(n,t)$ is given by the set 
		\[\left\{\sum_{x\in I(r,l)}\mu(y,x)x:y\in M(n,t)\right\}\]
	Each of these basis elements has a unique non-zero term of the form $y$ for each $y\in M(n,t)$ in its sum. All other non-zero terms are $x$ for $x$ strictly greater than $y$.\label{prop:dw4.3}
\end{prop}

We have an action of $\sym_n$ on the left of $I(n,t)$ by permuting the $n$ northern nodes, and an action of $\sym_t$ on the right by permuting the $t$ leftmost southern nodes. This gives a $(\sym_n,\sym_t)$-bimodule structure on $\Psi(n,t)$. Let $\sigma\in\sym_n$ and $x,y\in I(n,t)$ such that $x\prec y$. Then $\sigma x\prec\sigma y$, since $\sigma y$ will be a refinement of the set-partition represented by $\sigma x$. Therefore $\sigma$ will take one basis element as given in Proposition \ref{prop:dw4.3} to another. Similarly for $\tau\in\sym_t$ we have $x\tau\prec y\tau$.

We then have the following decomposition of $\Psi(n,t)$ as a $(\sym_n,\sym_t)$-bimodule.

\begin{prop}[{\cite[Proposition 4.4]{doran2000partition}}]
	As a $(\sym_n,\sym_t)$-bimodule,
		\[\Psi(n,t)\cong\biguplus_{\mu\vdash t}\left(\ind_{\sym_t\times \sym_{n-t}}^{\sym_n}(S_k^\mu\boxtimes 1_{\sym_{n-l}})\boxtimes S_k^\mu\right)\]
\end{prop}
	\begin{proof}
		By Proposition \ref{prop:dw4.3}, we can index a basis of $\Psi(n,t)$ by $M(n,t)$. Note that $\sym_n\times \sym_t$ acts transitively on this set. Let $y\in M(n,t)$ be the element below\\
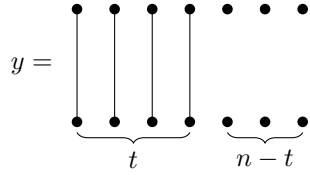
				\begin{center}
				\begin{tikzpicture}[scale=0.5]
					\foreach \x in {1,...,7}
					{\fill[black] (\x,3) circle (4pt);
					\fill[black] (\x,0) circle (4pt);}
					\foreach \x in {1,...,4}
					\draw (\x,0)--+(0,3);
					
					\draw (-0.2,1.5) node {$y=$};
					\draw[decorate,decoration={brace,mirror,raise=4pt,amplitude=4pt}] (1,0)--+(3,0);
					\draw[decorate,decoration={brace,mirror,raise=4pt,amplitude=4pt}] (5,0)--+(2,0);
					\draw (2.5,-1) node {$t$};\draw (6,-1) node {$n-t$};
				\end{tikzpicture}
				\end{center}
		We see that $\sym_t\times \sym_{n-t}$ is a natural subgroup of $\sym_n$, with $\sym_t$ acting on the leftmost $t$ northern nodes and $\sym_{n-t}$ acting on the remaining northern nodes. Then the stabiliser of $y$ in $\sym_n\times \sym_t$ is the set of permutations 
		\[H=\{((\gamma,\pi),\gamma^{-1}):\pi\in\sym_{n-t},\gamma\in\sym_t\}\subseteq\sym_n\times \sym_t\]
		Since the action of $\sym_n\times\sym_t$ on $M(n,t)$ is transitive, we can write $\Psi(n,t)=\ind_H^{\sym_n\times \sym_t}1_H$. We induce first to the subgroup $(\sym_t\times \sym_{n-t})\times \sym_t$ of $\sym_n\times \sym_t$. It is clear using Frobenius reciprocity that inducing the trivial module from the subgroup $L=\{(\gamma,\gamma^{-1}):\gamma\in\sym_t\}$ of $\sym_t\times \sym_t$ to $\sym_t\times \sym_t$ gives a module with filtration $\displaystyle\biguplus_{\mu\vdash t}(S_k^\mu\boxtimes S_k^\mu)$. Since $\sym_{n-t}$ has no effect here, it follows that
			\[\ind_H^{(\sym_t\times \sym_{n-t})\times \sym_t}1_H\cong\biguplus_{\mu\vdash t}\left(S_k^\mu\boxtimes1_{\sym_{n-t}}\right)\boxtimes S_k^\mu\]
Inducing the left side of the tensor product to $\sym_n$ then gives the required result. This has no effect on the last factor, as seen for example by taking coset representatives in $\sym_n$.
	\end{proof}

Using the Littlewood-Richardson rule we obtain the following decomposition.

\begin{prop}[{\cite[Proposition 4.5]{doran2000partition}}]
	As a $\sym_n\times \sym_t$-module
		\[\Psi(n,t)\cong\biguplus_{\substack{\lambda\vdash n, \mu\vdash t\\\text{with }c^\lambda_{\mu,(n-t)}=1}}S_k^\lambda\boxtimes S_k^\mu\]
	for all $\mu\vdash t$ and for a given $\mu$ only those $\lambda$ for which $c_{\mu,(n-t)}^\lambda\neq0$. These are the $\lambda$ which can be obtained from $\mu$ by adding $n-t$ nodes, no two in a column. \label{prop:dw4.5}
\end{prop}
	\begin{proof}
		This follows from the Littlewood-Richardson rule, generalised to arbitrary field by James and Peel in \cite{james1979specht}. The Littlewood-Richardson coefficients $c_{\mu,(n-t)}^\lambda$ can only be $0$ or $1$. Note that as a $\sym_n\times \sym_t$-module it is multiplicity free.
	\end{proof}

\begin{prop}[{\cite[Proposition 4.6]{doran2000partition}}]
	The submodule of $\Delta_\mu^k(n)$ which is annihilated by all $p_{i,j}$ is spanned by elements of the form $u\otimes s$ for $s\in S^\mu_k$ and $u\in\Psi(n,t)$.	\label{prop:dw4.6}
\end{prop}
	\begin{proof}
		First, let $u\in\Psi(n,t)$. Then for all $1\leq i<j\leq n$ and $s\in S^\mu_k$,
			\[p_{i,j}(u\otimes s)=(p_{i,j}u)\otimes s=0\otimes s=0\]
		So if $w\in\Psi(n,t)\otimes S^\mu_k$, then $p_{i,j}w=0$.
		 
		Conversely, suppose $p_{i,j}w=0$ for some $w\in V(n,t)\otimes S^\mu_k\cong\Delta_\mu^k(n)$. We can write
			\[w=\sum_{x\in I(n,t)}c_xx\otimes s_x\]
		where $c_x\in k$ and $s_x\in S^\mu_k\backslash\{0\}$. We know from \cite[Proposition 4.2]{doran2000partition} that if
			\[w'=\sum_{x\in I(n,t)}\sum_{y\in M(n,t)}c_y\mu(y,x)(x\otimes s_y)\]	
		then $p_{i,j}w'=0$, and therefore $p_{i,j}(w-w')=0$. Now
		\begin{align*}
			w-w'&=\sum_{x\in I(n,t)}c_xx\otimes s_x-\sum_{x\in I(n,t)}\sum_{y\in M(n,t)}c_y\mu(y,x)(x\otimes s_y)\\
				&=\sum_{x\in I(n,t)}x\otimes\left(c_xs_x-\sum_{y\in M(n,t)}c_y\mu(y,x)s_y\right)\\
		\end{align*}
		For $x,y\in M(n,t)$ we have $\mu(x,y)=1$ if $x=y$ and zero otherwise. Therefore there are no terms of the form $y\otimes s_y$ ($y\in M(n,t)$) in $w-w'$, and we may write
			\[w-w'=\sum_{x\in I(n,t)\backslash M(n,t)}d_xx\otimes t_x\]
		where $d_x\in k$ and $t_x\in S^\mu_k\backslash\{0\}$.
		
		If $w-w'\neq0$, then choose $x\in I(n,t)\backslash M(n,t)$ minimal with respect to $\prec$ for which $d_x\neq0$. As in the proof of \cite[Theorem 4.2]{doran2000partition} there exist $i,j$ for which $p_{i,j}x=x$, and therefore $p_{i,j}(x\otimes t_x)=x\otimes t_x$. Since $p_{i,j}$ can only join blocks of a diagram and $x$ is minimal with respect to $\prec$, we then see that $p_{i,j}(w-w')\neq0$, a contradiction.
		
		Therefore $w=w'$, and by Proposition \ref{prop:dw4.3} we see that $w\in\Psi(n,t)\otimes S^\mu_k$.
	\end{proof}
	
The following is a very restricted case of \cite[Proposition 4.7]{doran2000partition}, but is necessary for later use. The proof of the original proposition does not generalise to fields of positive characteristic.
	
\begin{prop}[{\cite[Proposition 4.7]{doran2000partition}}]
	Let $\mu$ be a partition with $|\mu|=t<p$, and suppose $\lambda\neq\mu$ is the only partition other than $\mu$ that appears as a composition factor of $\Delta_\mu^k(n)$. Then $\mu\subset\lambda$, all of the nodes in $[\lambda]/[\mu]$ are in different columns, and in fact $[\Delta_\mu^k(n):L_\lambda^k(n)]=1$.\label{prop:dw4.7}
\end{prop}
	\begin{proof}
		By localising we may assume that $\lambda\vdash n$. By the cellularity of $P_n^k(\delta)$ we see that $L_\mu^k(n)$ appears precisely once as a composition factor of $\Delta_\mu^k(n)$, as the head of the module. Therefore $\Delta_\mu^k(n)$ has structure
		\begin{center}
			$L_\mu^k(n)$\\[5pt]
			$\displaystyle\biguplus L_\lambda^k(n)$
		\end{center}
	Thus there is a submodule $W\subset \Delta_\mu^k(n)$ isomorphic to $\biguplus L_\lambda^k(n)$, and therefore a sequence of modules
		\[0=W_0\subset W_1\subset W_2\subset\dots\subset W_{r-1}\subset W_r=W\]
		such that $W_i/W_{i-1}\cong L_\lambda^k(n)$ for $1\leq i\leq r$. Let $w_r\in W_r=W$, and consider $p_{i,j}w_r$. Since \mbox{$W_r/W_{r-1}\cong L_\lambda^k(n)$} is a module for the symmetric group, it must be annihilated by $p_{i,j}$. Therefore $p_{i,j}w_r=w_{r-1}$ for some $w_{r-1}\in W_{r-1}$. By the same argument, we also see that $p_{i,j}w_{r-1}=w_{r-2}$ for some $w_{r-2}\in W_{r-2}$, and so $p_{i,j}^2w_r=w_{r-2}$. Repeating this process we arrive at $p_{i,j}^rw_r=0$, and since $p_{i,j}$ is an idempotent we deduce that $p_{i,j}w_r=0$ for all $w_r\in W$. By Proposition \ref{prop:dw4.6}, $W$ must then be in $\Psi(n,t)\otimes S^\mu_k$. 
		
		Consider now the module $W_1\cong L_\lambda^k(n)$. Since $|\mu|=t<p$, as a left $k\sym_n$-module it is $S^\lambda_k$ and we can find an idempotent $e_\mu$ such that $S^\mu_k=k\sym_te_\mu$. Then for $\tau\neq\mu$, $|\tau|=|\mu|$, we have $e_\tau k\sym_te_\mu=0$ and so $e_\tau k\sym_t\otimes_{\sym_t}k\sym_te_\mu=0$. Therefore $(S^\lambda_k\boxtimes S^\tau_k)\otimes_{\sym_t} S^\mu_k=0$ if $\tau\neq\mu$. This means the only terms from Proposition \ref{prop:dw4.5} we need consider in $\Psi(n,t)$ are those $\{\lambda,\mu\}$ with this given $\mu$. By Proposition \ref{prop:dw4.5}, $\mu\subset\lambda$ and the nodes of $[\lambda]/[\mu]$ are in different columns. Furthermore, $c^\lambda_{\mu,(n-t)}=1$, and so there is a unique copy.
	\end{proof}

In the rest of this section, we will consider separately different cases concerning the values of $n$ and $\delta$. The first distinction we make is due to the following Lemma:
\begin{lem}[{\cite[Corollary 5.8]{bowmanblocks}}]\label{lem:deltapcondition}
	Suppose there exist partitions $\lambda\vdash n$, $\mu\vdash n-t$ $(t>0)$ with $[\Delta_\mu^k(n;\delta):L_\lambda^k(n;\delta)]\neq0$. Then $\delta\in\mathbb{F}_p$, the prime subfield of $k$.
\end{lem}

\subsection{$\delta\not\in\mathbb{F}_p$}

We will show that in this case the decomposition matrix $\mathbf{D}(P_n^k(\delta))$ is equal to a block diagonal matrix, with components equal to the decomposition matrices of symmetric group algebras. A proof of this result can also be found in \cite[Corollary 6.2]{hartmann2010cohomological}.

\begin{thm}[{\cite[Corollary 6.2]{hartmann2010cohomological}}]\label{thm:r>pdeltanot}
	Suppose $\delta\not\in\mathbb{F}_p$. Then the decomposition matrix $\mathbf{D}(P_n^k(\delta))$ is equal to the block diagonal matrix
		\begin{equation}\label{eq:blockdecomp}\mathbf{S}=\left(\begin{array}{c@{}ccc@{}c}
		\setlength{\fboxsep}{20pt}\framebox{$\mathbf{D}(k\sym_n)$}&&&&\\
		&\setlength{\fboxsep}{15pt}\framebox{$\mathbf{D}(k\sym_{n-1})$}&&&\\
		&&\ddots&&\\
		&&&\setlength{\fboxsep}{5pt}\framebox{$\mathbf{D}(k\sym_1)$}&\\
		&&&&\setlength{\fboxsep}{3pt}\framebox{$\mathbf{D}(k\sym_0)$}
		\end{array}
	\right)\end{equation}
\end{thm}
	\begin{proof}
		By the cellularity of $P_n^k(\delta)$ we immediately see that $[\Delta_\mu^k(n):L_\lambda^k(n)]=0$ if $|\lambda|<|\mu|$.

		If $|\lambda|>|\mu|$, then by Lemma \ref{lem:deltapcondition} we see that as $\delta\not\in\mathbb{F}_p$, the decomposition number \mbox{$[\Delta_\mu^k(n;\delta):L_\lambda^k(n;\delta)]$} must be zero.
		
		If now $|\lambda|=|\mu|$, then by localising we have
			\[[\Delta_\mu^k(n):L_\lambda^k(n)]=[S^\mu_k:D^\lambda_k]\]
		and the result follows as these are the entries of the decomposition matrix of $k\sym_{|\lambda|}$.
	\end{proof}
		
\subsection{$n<p$ and $\delta\in\mathbb{F}_p$}

We will see that in this case, any non-zero decomposition numbers arise from reducing homomorphisms in the characteristic zero case of the partition algebra $P_n^K(\delta+rp)$ for some $r\in\mathbb{Z}$.

\begin{lem}\label{lem:r<pdeltain}
	Let $n<p$ and $\delta\in\mathbb{F}_p$. If $[\Delta_\mu^k(n;\delta):L_\lambda^k(n;\delta)]\neq0$ then either $\lambda=\mu$ or $\mu\hookrightarrow_{\delta+rp}\lambda$ for a unique $r\in\mathbb{Z}$.
\end{lem}
	\begin{proof}
		By localising we may assume that $\lambda\vdash n$. We will prove this result by induction on $n$. If $n=0$ then we have $\lambda=\mu=\emptyset$ and the result clearly holds by the cellularity of $P_0^k(\delta)$.
		
		Since $n<p$, we must have $L_\lambda^k(n;\delta)=\Delta_\lambda^k(n;\delta)\cong S_k^\lambda$, the Specht module. If we apply the restriction functor to this module, then by the branching rule the result is non-zero:
			\[\mathrm{res}_nL_\lambda^k(n;\delta) =\mathrm{res}_n\Delta_\lambda^k(n;\delta)\cong\biguplus_{\nu\triangleleft\lambda}\Delta_\nu^k(n-\textstyle\frac{1}{2};\delta)\]
			Since $\nu\vdash n-1$ all the modules in this filtration are Specht modules, and since $n<p$ they are also simple.
		Thus $[\Delta_\mu^k(n;\delta):L_\lambda^k(n;\delta)]\neq0 \implies [\mathrm{res}_n\Delta_\mu^k(n;\delta):L_\nu^k(n-\frac{1}{2};\delta)]\neq0$ for some $\nu\triangleleft\lambda$. Recall that we have an exact sequence
			\[0\longrightarrow\biguplus_{\eta\triangleleft\mu}\Delta_\eta^k(n-\textstyle\frac{1}{2};\delta) \longrightarrow\mathrm{res}_n\Delta_\mu^k(n)\longrightarrow\Delta_\mu^k(n-\frac{1}{2};\delta)\longrightarrow0\]
		and therefore a filtration of $\mathrm{res}_n\Delta_\mu^k(n;\delta)$ by modules $\Delta_\eta^k(n-\frac{1}{2};\delta)$ with $\eta\triangleleft\mu$ or $\eta=\mu$.
		
		Using the Morita equivalence from Proposition \ref{prop:morita} we must therefore have\linebreak \mbox{$[\Delta_\eta^k(n-1;\delta-1):L_\nu^k(n-1;\delta-1)]\neq0$}. So by induction on $n$, $\eta$ and $\nu$ must be a $(\delta-1+rp)$-pair for some $r\in\mathbb{Z}$.
		
		Let $\eta=(\eta_1,\dots,\eta_m)$. Then there is some $i$ such that
			\[\nu=(\eta_1,\dots,\eta_{i-1},\delta-1+rp-|\eta|+i,\eta_{i+1},\dots,\eta_m)\]
		Suppose first that $\eta=\mu$. There is some $j$ such that $\lambda=\nu+\varepsilon_j$. If $j=i$ then
			\begin{align*}
				\lambda&=(\nu_1,\dots,\nu_{i-1},\nu_i+1,\nu_{i+1},\dots,\nu_m)\\
						&=(\eta_1,\dots,\eta_{i-1},\delta+rp-|\eta|+i,\eta_{i+1},\dots,\eta_m)\\
						&=(\mu_1,\dots,\mu_{i-1},\delta+rp-|\mu|+i,\mu_{i+1},\dots,\mu_m)
			\end{align*}
		and so $\mu\hookrightarrow_{\delta+rp}\lambda$. If $j\neq i$, then since $\lambda$ and $\mu$ are in the same $k$-block we must have $\hat\mu+\rho(\delta)\sim_p\hat\lambda+\rho(\delta)$. We calculate
		\begin{align*}
			|\lambda|&=|\mu|-\mu_i+(\delta-1+rp-|\mu|+i)+1\\
					&=\delta+rp-\mu_i+i
		\end{align*}
	and so
		\[\hat\lambda+\rho(\delta)=(\mu_i-i-tp,\mu_1-1,\mu_2-2,\dots,\mu_j+1-j,\dots,\delta-1+rp-|\mu|,\dots,\mu_m-m)\]
	By pairing equal elements from $\hat\lambda+\rho(\delta)$ and $\hat\mu+\rho(\delta)$ we are left with
		\[\{\mu_i-i-tp,\mu_j+1-j,\delta-1+rp-|\mu|\}\sim_p\{\mu_i-i,\mu_j-j,\delta-|\mu|\}\]
	Clearly $\mu_i-i-rp\equiv\mu_i-i$ (mod $p$), so we must have $\delta-1+rp-|\mu|\equiv\mu_j-j$ (mod $p$). These are the contents of the final node in row $i$ of $\lambda$ and the penultimate node in row $j$ respectively, and since $|\lambda|=r<p$ these cannot differ by $p$ or more. Hence $\delta-1+rp-|\mu|=\mu_j-j$. But this cannot be true as these are the contents of the final nodes in different rows of $\nu$.
	\\~\\
	Suppose now that $\eta\triangleleft\mu$, so that $\mu=\eta+\varepsilon_k$.
	
	If $k=j=i$, that is we are adding nodes to the same row to obtain $\mu$, $\nu$ and $\lambda$, then we have
	\begin{align*}
		\hat\mu+\rho(\delta)&=(\delta-|\eta|-1,\eta_1-1,\dots,\eta_i+1-i,\dots,\eta_m-m)\\
		\hat\lambda+\rho(\delta)&=(\eta_i-i,\eta_1-1,\dots,\delta+rp-|\eta|,\dots,\eta_m-m)
	\end{align*}
	Since $\hat\mu+\rho(\delta)\sim_p\hat\lambda+\rho(\delta)$ and $p>2$, we must have $\delta-|\eta|-1\equiv\eta_i-i$ (mod $p$). By a similar argument involving contents to above, this must be an equality, and implies that $\eta=\nu$. Hence $\mu=\lambda$.
		 
		 If $k=i$, $j\neq i$ then we have
	\begin{align*}
		\hat\mu+\rho(\delta)&=(\delta-|\eta|-1,\eta_1-1,\dots,\eta_i+1-i,\dots,\eta_m-m)\\
		\hat\lambda+\rho(\delta)&=(\eta_i-i,\eta_1-1,\dots,\eta_j+1-i,\dots,\delta+rp-1-|\eta|,\dots,\eta_m-m)
	\end{align*}
	and arguing as above this results in $\eta_i-i=\eta_j-j$, which is impossible.
	
	If now we suppose $k\neq i$, $j=i$, then
	\begin{align*}
		\hat\mu+\rho(\delta)&=(\delta-|\eta|-1,\eta_1-1,\dots,\eta_k+1-k,\dots,\eta_m-m)\\
		\hat\lambda+\rho(\delta)&=(\eta_i-i,\eta_1-1,\dots,,\delta+rp-|\eta|,\dots,\eta_m-m)
	\end{align*}
	so that $\eta_k-k+1=\delta+tp-|\eta|$ which is again impossible.
	
	Finally suppose $k\neq i$ and $j\neq i$. We have
	\begin{align*}
		\hat\mu+\rho(\delta)&=(\delta-|\eta|-1,\eta_1-1,\dots,\eta_k+1-k,\dots,\eta_m-m)\\
		\hat\lambda+\rho(\delta)&=(\eta_i-i,\eta_1-1,\dots,\eta_j+1-i,\dots,\delta+rp-1-|\eta|,\dots,\eta_m-m)
	\end{align*}
	and hence $\eta_k-k=\eta_j-j$, which is impossible unless $k=j$ and thus $\mu\hookrightarrow_{\delta+rp}\lambda$.
	\end{proof}
	
\begin{thm}\label{thm:r<pdeltain}
	Let $n<p$, $\delta\in\mathbb{F}_p$ and suppose $\mu\in\Lambda_{\leq n}$ is such that $\Delta_\mu^k(n;\delta)\neq L_\mu^k(n;\delta)$. Then there is a unique $r\in\mathbb{Z}$ such that $[\Delta_\mu^k(n;\delta):L_\lambda^k(n;\delta)]=[\Delta_\mu^K(n;\delta+rp):L_\lambda^K(n;\delta+rp)]$ for all $\lambda\in\Lambda_{\leq n}$. That is, $\Delta_\mu^k(n;\delta)$ has Loewy structure
	\begin{center}
		$L_{\mu}^k(n;\delta)$\\
		$L_{\lambda}^k(n;\delta)$
	\end{center}
	for a unique $\lambda$ such that $\mu\hookrightarrow_{\delta+rp}\lambda$.
\end{thm}
	\begin{proof}
		Since $\Delta_\mu^k(n;\delta)\neq L_\mu^k(n;\delta)$ there is some $\lambda\neq\mu$ such that $[\Delta_\mu^k(n;\delta):L_\lambda^k(n;\delta)]\neq0$. By Lemma \ref{lem:r<pdeltain}, there exists a unique $r\in\mathbb{Z}$ such that $\mu\hookrightarrow_{\delta+rp}\lambda$.
		
		Suppose now there is another partition $\nu\neq\lambda,\mu$ such that $[\Delta_\mu^k(n;\delta):L_\lambda^k(n;\delta)]\neq0$. Again there is a unique $r'\in\mathbb{Z}$ such that $\mu\hookrightarrow_{\delta+r'p}\nu$. We will show that this leads to a contradiction.
		
		Consider first the case $r=r'$. Since both $\lambda$ and $\nu$ are obtained from $\mu$ by adding a single row of nodes, the final node having content $\delta+rp-|\mu|$, we immediately see that we cannot be adding nodes to the same row, otherwise $\lambda=\nu$. So suppose we add nodes to row $i$ to obtain $\lambda$ and to $j$ to obtain $\nu$, with $i<j$. Then $\nu_j-j=\delta+rp-|\mu|$, and since $\nu$ is a partition we must have $\nu_m-m=\mu_m-m>\delta+rp-|\mu|$ for all $m<j$. In particular $\mu_i-i>\delta+rp-|\mu|$, and so we cannot add nodes to this row to obtain $\lambda$.
		
		Suppose now that $r\neq r'$. Assume again that we are adding nodes to row $i$ to obtain $\lambda$, and to row $j$ to obtain $\nu$, with $i<j$. Therefore $\lambda_i-i=\delta+rp-|\mu|$ and $\nu_j-j=\delta+r'p-|\mu|$. Notice that
		\begin{align*}
			\delta+rp-|\mu|&=\lambda_i-i\\&>\mu_i-i\\&=\nu_i-i\\&>\nu_j-j\\&=\delta+r'p-|\mu|
		\end{align*}
		and hence $r>r'$.
		
		The hook in the Young diagram $[\lambda]\cup[\mu]$ with endpoints the last nodes of rows $i$ and $j$ contains $(r-r')p+1$ nodes. Since $\lambda$ and $\nu$ differ only in rows $i$ and $j$, the part of this hook lying inside $[\lambda]$ contains $(r-r')p$ nodes. Therefore $|\lambda|\geq(r-r')p$, which cannot happen if $n<p$.
		
		We have therefore shown that there cannot be two distinct partitions that appear as a composition factor of $\Delta_\mu^k(n;\delta)$ (other than $\mu$ itself). Thus we can apply Proposition \ref{prop:dw4.7} to see that \mbox{$[\Delta_\mu^k(n;\delta):L_\lambda^k(n;\delta)]=1$}, and the result follows.
	\end{proof}

\begin{rem}\label{rem:glue}
	Theorem \ref{thm:r<pdeltain} shows us that the decomposition matrix of $P_n^k(\delta)$ when $n<p$ and $\delta\in\mathbb{F}_p$ is obtained by ``putting together'' all of the characteristic zero decomposition matrices for each lift of $\delta$ to $K$.
\end{rem}

\subsection{Case $3$: $n\geq p$ and $\delta\in\mathbb{F}_p$}

In this case, the decomposition matrix of the partition algebra $P_n^k(\delta)$ is much more complicated. However there is still one sub-case when we can give a complete description.

\begin{lem}
	Let $n\geq p$ and $\delta\in\mathbb{F}_p$. Then there is only one lift of $\delta\in \mathbb{F}_p$ to $R$ such that the partition algebra $P_n^K(\delta)$ is non-semisimple if and only if $n=p$ and $\delta=p-1$.
\end{lem}
	\begin{proof}
		First notice that if $\delta<0$ then $P_n^K(\delta)$ is always semisimple, as we can never have a $\delta$-pair $\mu\hookrightarrow_\delta\lambda$. Combining this with \cite[Theorem 3.27]{halverson2005partition} we see that $P_r^K(\delta)$ is non-semisimple if and only if $0\leq\delta<2n-1$.
		
		Suppose that $n=p$ and $\delta=p-1$. The lifts of $\delta$ to $R$ are $p-1+rp$ with $r\in\mathbb{Z}$. Clearly there is only one value of $r$ such that $0\leq p-1+rp<2p-1$, namely $r=0$, and so there is only one lift of $\delta$ producing a non-semisimple algebra.
		
		Suppose now that $n>p$. Again, the lifts of $\delta$ to $K$ are $\delta+rp$, only now we have at least two values of $r$ such that $0\leq\delta+rp<2n-1$. Clearly $r=0$ satisfies this, but also since $\delta\leq p-1$ so too does $r=1$. Thus we have more than one lift of $\delta$ giving a non-semisimple algebra.
		
		Finally, suppose $\delta\neq p-1$. Once more, the lifts of $\delta$ are $\delta+rp$, and we want to satisfy the condition $0\leq\delta+rp<2n-1$. We have $r=0$ immediately, and since $\delta<p-1<n$ we can also choose $r=1$, so again there is more than one lift of resulting in a non-semisimple algebra.
	\end{proof}
	
Because of this result, we will henceforth restrict our attention to the the case $n=p$ and $\delta=p-1$. We continue by first calculating the decomposition numbers for the block $\mathcal{B}_\emptyset^k(p;p-1)$.

\begin{lem}\label{lem:principalblock}
	The block $\mathcal{B}_\emptyset^k(p;p-1)$ contains precisely all partitions with empty $p$-core.
\end{lem}
	\begin{proof}
		Using Theorem \ref{thm:pblocks} we look instead at the orbit $\mathcal{O}_\emptyset^p(p;p-1)$, and characterise the partitions therein. This is accomplished by constructing the marked abacus of $\emptyset$ using $p$ beads. The number of beads on each runner is given by $\Gamma(\emptyset,p)=(1,1,\dots,1,1)$. The runner $v_\emptyset$ is given by the $p$-congruence class of $\delta-|\emptyset|+p\equiv p-1$, so we therefore have $\Gamma_{p-1}(\emptyset,p)=(1,1,\dots,1,2)$. The block $\mathcal{B}^k_\emptyset(p;p-1)$ thus contains all partitions $\lambda$ with $\Gamma_{p-1}(\lambda,p)=(1,1,\dots,1,2)$.

		Let $\lambda$ be such a partition. If $v_\lambda=p-1$, then $\Gamma(\lambda,p)=(1,1,\dots,1,1)$ and so $\lambda$ has empty $p$-core. If $v_\lambda=m$ for some $0\leq m<p-1$, then 
		\[\Gamma(\lambda,p)=(1,\dots,1,\underbrace{0}_{\substack{(m+1)\text{-th}\\\text{place}}},1,\dots,1,2)\]
		Now let $\mu$ be the $p$-core of $\lambda$. Note that we must have $|\mu|\leq|\lambda|$. Since $\Gamma(\mu,p)=\Gamma(\lambda,p)$ and all beads are as high up their runners as possible, we can find $\mu$ explicitly. First we see that
			\[\beta(\mu,p)=(2p-1,p-1,p-2,\dots,m+1,m-1,m-2,\dots,2,1,0)\]
	and therefore
			\begin{align*}
				\mu&=\beta(\mu,p)+(1,2,\dots,p)-(p,p,\dots,p)\\
					&=(p,1,1,\dots,1,0,0,\dots,0)
			\end{align*}
		It is then clear that $|\mu|>p$. Since $|\mu|\leq|\lambda|$ we see that $|\lambda|>p$ and therefore $\lambda$ cannot label a $P_p^k(p-1)$ cell module.
		\\~\\
		Conversely, if $\lambda\vdash t\leq p$ and has empty $p$-core, then $|\lambda|=p$ or $0$, and
		\begin{align*}
			\Gamma_{p-1}(\lambda,p)&=(1,1,\dots,1,2)\\
			&=\Gamma_{p-1}(\emptyset,p)
		\end{align*}
Therefore $\lambda\in\mathcal{B}^k_\emptyset(p;p-1)$.
	\end{proof}
	
	Having determined which partitions lie in $\mathcal{B}_\emptyset^k(p;p-1)$, we will now determine the decomposition matrix of this block.
	
\begin{lem}\label{lem:empty}
	The composition series of $\Delta_\emptyset^k(p;p-1)$ is $0\subset\Delta_{(p)}^k(p;p-1)\subset\Delta_\emptyset^k(p;p-1)$.
\end{lem}
	\begin{proof}
		Firstly, $\emptyset\hookrightarrow_{p-1}(p)$ since the partitions differ in one row only and the final node of this row of $(p)$ has content $p-1=\delta-|\emptyset|$. Thus there is a non-trivial homomorphism $\Delta_{(p)}^K(p;p-1)\longrightarrow\Delta_\emptyset^K(p;p-1)$, and so we must have $\mathrm{Hom}(\Delta_{(p)}^k(p;p-1),\Delta_\emptyset^k(p;p-1))\neq0$ by Lemma \ref{lem:modred}.
		
		We must now show that there is no module $N$ such that $\Delta_{(p)}^k(p;p-1)\subsetneq N\subsetneq\Delta_\emptyset^k(p;p-1)$. By Lemma \ref{lem:principalblock} and the cellularity of $P_p^k(p-1)$, any such module $N$ would be a symmetric group module. In particular, the action of any element $p_{i,j}$ on $N$ must be zero, and by Proposition \ref{prop:dw4.6} we see that $N\subseteq\Psi(p,0)\otimes S^\emptyset_k\cong\Psi(p,0)$. Now from Proposition \ref{prop:dw4.3} we have a basis for $\Psi(p,0)$ given by the set
		\[\left\{\sum_{x\in I(p,0)}\mu(y,x)x:y\in M(p,0)\right\}\]
		But $M(p,0)$ consists of only one element, namely the diagram with each node in its own block. Therefore the module $\Psi(p,0)$ is one-dimensional and is isomorphic to $\Delta_{(p)}^k(p;p-1)$. Thus there can be no module $N$ with $\Delta_{(p)}^k(p;p-1)\subsetneq N\subsetneq\Delta_\emptyset^k(p;p-1)$.
	\end{proof}
	
	From Lemma \ref{lem:principalblock} we have $(p)\in\mathcal{B}_\emptyset^k(p;p-1)$, and Lemma \ref{lem:empty} shows us that in fact\linebreak \mbox{$[\Delta^k_\emptyset(p;p-1):L_{(p)}^k(p;p-1)]=1$}. The remaining partitions in this block are all $p$-hook partitions, i.e. are of the form $(p-m,1^m)$ for some $1<m<p-1$, since these are the only partitions of $p$ with empty $p$-core. Because of this, the following result from Peel allows us to complete our description of the decomposition matrix of the block $\mathcal{B}_\emptyset^k(p;p-1)$:
	
\begin{thm}[{\cite[Theorem 1]{peel1971hook}}]\label{thm:peel}
	Let $\mathrm{char}\;k=p>2$. A composition series for $S_k^{(p-m,1^m)}$, $0<m<p-1$, is given by
	\[0\subset\mathrm{Im}\;\theta^{m-1}\subset S_k^{(p-m,1^m)}\]
	where $\theta^{m-1}:S_k^{(p-(m-1),1^{m-1})}\longrightarrow S_k^{(p-m,1^m)}$ is a non-trivial $k\sym_p$-homomorphism. Furthermore, $S_k^{(1^p)}\cong S_k^{(2,1^{p-2})}/\mathrm{Im}\;\theta^{p-3}$.
\end{thm}
\begin{cor}\label{cor:peel}
Let $\mathrm{char}\;k=p>2$. For $0<m<p-1$ we have
	\[[\Delta_{(p-m,1^m)}^k(p;p-1):L_\lambda^k(p;p-1)]=\begin{cases}
													1&\text{if }\lambda=(p-m,1^m)\text{ or }(p-m+1,1^{m-1})\\
													0&\text{otherwise}
												\end{cases}\]
	For $m=0$ we have $\Delta_{(p)}^k(p;p-1)\cong L_{(p)}^k(p;p-1)$.\\
	For $m=p-1$ we have $\Delta_{(1^p)}^k(p;p-1)\cong L_{(2,1^{p-2})}^k(p;p-1)$.
\end{cor}
\begin{proof}
	We apply Theorem \ref{thm:partsymblocks} to Theorem \ref{thm:peel}. For $m=0$ we use the fact that $P_p^k(p-1)$ is a cellular algebra.
\end{proof}
We now turn our attention to the other blocks of $P_p^k(p-1)$. The partitions here must have non-empty $p$-core, and since all partitions have size at most $p$, then they are themselves all $p$-cores.
\begin{lem}\label{lem:sizes}
	Let $\lambda\in\Lambda_{\leq p}$ be a partition with non-empty $p$-core. If there exists $\mu\in\mathcal{B}_\lambda^k(p;p-1)\backslash\{\lambda\}$, then $|\lambda|\neq|\mu|$.
\end{lem}
	\begin{proof}
		Choose a partition $\mu\in\mathcal{B}_\lambda^k(p;p-1)$ with $|\lambda|=|\mu|$. By the characterisation of the blocks of the partition algebra we have $\Gamma_{p-1}(\lambda,p)=\Gamma_{p-1}(\mu,p)$. As $|\lambda|=|\mu|$ we have $v_\lambda=v_\mu$, hence $\Gamma(\lambda,p)=\Gamma(\mu,p)$ and they have the same $p$-core. However since $|\lambda|\leq p$ and has non-empty $p$-core, it must in fact be that $p$-core. Since we also have $|\mu|\leq p$, it follows that $\mu=\lambda$.
	\end{proof}
	
\begin{thm}\label{thm:r=p}
	Let $\lambda\in\Lambda_{\leq p}$ be a partition with non-empty $p$-core. Then the block $\mathcal{B}_\lambda^k(p;p-1)$ has the same decomposition matrix as $\mathcal{B}_\lambda^K(p;p-1)$.
\end{thm}
	\begin{proof}
		By Lemma \ref{lem:sizes} we can relabel
			\[\mathcal{B}_\lambda^k(p;p-1)=\{\lambda^{(m)},\lambda^{(m-1)},\dots,\lambda^{(1)}\}\]
		where $|\lambda^{(i)}|>|\lambda^{(i-1)}|$ for $1<i\leq m$.
		
		Suppose $|\lambda^{(m)}|\neq p$. Then every partition in the block has size strictly less than $p$, and so labels a cell module for $P_{p-1}^k(p-1)$. Since the partition algebras form a tower of recollement, the decomposition matrix of the block $\mathcal{B}_\lambda^k(p;p-1)$ is the same as that of $\mathcal{B}_{\lambda}^k(p-1;p-1)$. We can therefore use the results of Theorem \ref{thm:r<pdeltain} to conclude that the decomposition numbers $[\Delta_{\lambda^{(i)}}^k(p-1;p-1):L_{\lambda^{(j)}}^k(p-1;p-1)]$ are either 0 or 1, and the latter occurs if and only if $\lambda^{(i)}\hookrightarrow_{p-1+rp}\lambda^{(j)}$ for some $r\in\mathbb{Z}$. But since $\delta=p-1$ is the only lift of $\delta$ to $K$ that gives a non-semisimple $K$-algebra, we must have $\lambda^{(i)}\hookrightarrow_{p-1}\lambda^{(j)}$. Therefore we have
			\begin{align*}
				[\Delta_{\lambda^{(i)}}^k(p;p-1):L_{\lambda^{(j)}}^k(p;p-1)]&= [\Delta_{\lambda^{(i)}}^k(p-1;p-1):L_{\lambda^{(j)}}^k(p-1;p-1)]\\
				&=[\Delta_{\lambda^{(i)}}^K(p-1;p-1):L_{\lambda^{(j)}}^K(p-1;p-1)]\\
				&=[\Delta_{\lambda^{(i)}}^K(p;p-1):L_{\lambda^{(j)}}^K(p;p-1)]
			\end{align*}
			
		Suppose now that $|\lambda^{(m)}|=p$. Then the partitions $\lambda^{(m-1)},\lambda^{(m-2)},\dots,\lambda^{(1)}$ are all of size strictly less than $p$, and therefore label cell modules for $P_{p-1}^k(p-1)$. By the same argument as above, the decomposition matrix obtained by removing the row and column labelled by $\lambda^{(m)}$ is the same as that of $\mathcal{B}_{\lambda^{(1)}}^k(p-1;p-1)$, which is the same as in characteristic zero.
		
		It remains to show that the decomposition numbers $[\Delta_{\lambda^{(i)}}^k(p;p-1):L_{\lambda^{(m)}}^k(p;p-1)]$ are the same as in characteristic zero. We begin by showing that $[\Delta_{\lambda^{(i)}}^k(p;p-1):L_{\lambda^{(m)}}^k(p;p-1)]=0$ for $i<m-1$. Since $\lambda^{(m)}$ is the only partition of size $p$ in its block, the simple module $L_{\lambda^{(m)}}^k(p;p-1)$ is a Specht module. Therefore after applying the restriction functor we have the following filtration:
		\[\mathrm{res}_pL_{\lambda^{(m)}}^k(p;p-1)\cong\biguplus_{\nu\triangleleft\lambda^{(m)}}\Delta_{\nu}^k(p-\textstyle\frac{1}{2};p-1)\]
Therefore we can apply the same argument as in Theorem \ref{thm:r<pdeltain} and see that if\linebreak \mbox{$[\Delta_{\lambda^{(i)}}^k(p;p-1):L_{\lambda^{(m)}}^k(p;p-1)]\neq0$}, then either $\lambda^{(i)}=\lambda^{(m)}$ or $\lambda^{(i)}\hookrightarrow_{p-1}\lambda^{(m)}$. Following the proof of Theorem \ref{thm:r<pdeltain} we must then have
\[[\Delta_{\lambda^{(i)}}^k(p;p-1):L_{\lambda^{(m)}}^k(p;p-1)]=\begin{cases}
												1&\text{if }i=m-1,m\\
												0&\text{otherwise}
												\end{cases}\]
and hence $[\Delta_{\lambda^{(i)}}^k(p;p-1):L_{\lambda^{(m)}}^k(p;p-1)]=[\Delta_{\lambda^{(i)}}^K(p;p-1):L_{\lambda^{(m)}}^K(p;p-1)]$ by Theorem \ref{thm:martinblocks}.
\end{proof}

\begin{rem}\label{rem:product}
	If we denote again by $\mathbf{S}$ the block decomposition matrix of the symmetric group algebras over $k$ (see \eqref{eq:blockdecomp}), then we can combine Lemma \ref{lem:empty}, Corollary \ref{cor:peel} and Theorem \ref{thm:r=p} and say that the decomposition matrix $\mathbf{D}(P^k_p(p-1))$ is equal to the product $\mathbf{D}(P^K_p(p-1))\,\mathbf{S}$. In fact, we can compute $\mathbf{S}$ explicitly in this case using Corollary \ref{cor:peel}.
\end{rem}

Unfortunately without the restrictions imposed thus far, we encounter examples of partition algebras whose decomposition matrices are not obtained from the methods summarised in Remarks \ref{rem:glue} or \ref{rem:product}. One such is detailed below.

\begin{example}
	We will show the decomposition  matrix of $P^k_4(1)$ with $\mathrm{char}\;k=3$ cannot be computed as in Remarks \ref{rem:glue} or \ref{rem:product}. We present below the decomposition matrix of $k\sym_4$.
		\[\bordermatrix{&D_k^{(4)}&D_k^{(3,1)}&D_k^{(2^2)}&D_k^{(2,1^2)}\cr
					  S^{(4)}_k   		&1&0&0&0\cr
					  S^{(3,1)}_k  		&0&1&0&0\cr
					  S^{(2^2)}_k  		&1&0&1&0\cr
					  S_k^{(2,1^2)}		&0&0&0&1\cr
					  S_k^{(1^4)}  		&0&0&1&0\cr}\]

	 We will first show that there exist non-zero decomposition numbers $[\Delta_\mu^k(4;1):L_\lambda^k(4;1)]$ for which there is no $r\in\mathbb{Z}$ such that $\mu\hookrightarrow_{1+3r}\lambda$, thus not following Remark \ref{rem:glue}. Indeed, examination of the decomposition matrix of $k\sym_4$ combined with Theorem \ref{thm:partsymblocks} shows us that $\Delta_{(2^2)}^k(4;1)$ has a submodule isomorphic to $L_{(4)}^k(4;1)$. Therefore $[\Delta_{(2^2)}^k(4;1):L_{(4)}^k(4;1)]\neq0$, but $(2^2)\not\subset(4)$ and so there cannot exist an integer $r$ with $(2^2)\hookrightarrow_{\delta+rp}(4)$.
	 
	We will now show that the decomposition matrix of $P^k_4(1)$ is not equal to the product of the decomposition matrices $\mathbf{D}(P^K_4(1+3r))\,\mathbf{S}$, for any $r\in\mathbb{Z}$. The semisimplicity criterion of \linebreak\cite[Theorem 3.27]{halverson2005partition} shows us that we must consider $r=0,1$.
		
	Consider first the case $r=0$, that is $P_4^K(1)$. We let $\lambda=(2,1^2)$ and $\mu=(2,1)$, then we have $\delta-|\mu|=-2$. Note that these partitions differ by a single node of content $-2$ in the third row, and therefore form a $1$-pair. By Theorem \ref{thm:martinblocks} we thus have $[\Delta_{(2,1)}^K(4;1):L_{(2,1^2)}^K(4;1)]=1$, and so by Lemma \ref{lem:modred}
	\begin{equation}\Delta_{(2,1)}^k(4;1):L_{(2,1^2)}^k(4;1)]\neq0\label{eq:partlift1}\end{equation}
	
	Now consider the case $r=1$, i.e. $P_4^K(4)$. Let $\lambda=(4)$ and $\mu=(1)$, then we have $\delta-|\mu|=3$. These partitions differ by a strip of nodes in the first row, the last of which has content $3$, and therefore form a $4$-pair. By Theorem \ref{thm:martinblocks} we see that $[\Delta_{(1)}^K(4;4):L_{(4)}^K(4;4)]=1$, and so by Lemma \ref{lem:modred}
		\begin{equation}[\Delta_{(1)}^k(4;1):L_{(4)}^k(4;1)]\neq0\label{eq:partlift4}\end{equation}

	 If the decomposition matrix $\mathbf{D}(P^k_4(1))$ was equal to the product $\mathbf{D}(P^K_4(1+3r))\,\mathbf{S}$ for some $r\in\mathbb{Z}$, we would have the following expansion for every $\mu\in\Lambda_{\leq4}$, $\lambda\in\Lambda^\ast_{\leq4}$:
	 \[[\Delta^k_\mu(4;1):L^k_\lambda(4;1)]=\sum_{\nu\in\Lambda_{\leq4}}[\Delta^K_\mu(4;1+3r):L^K_\nu(4;1+3r)][S^\nu_k:D^\lambda_k]\]
	 First let $r=0$, $\lambda=(4)$ and $\mu=(1)$. By examining the decomposition matrix of $k\sym_4$, we see that the only partitions $\nu$ for which $[S^\nu_k:D^{(4)}_k]\neq0$ are $\nu=(4)$ and $\nu=(2^2)$. The above factorisation then becomes
	 \begin{align*}
	 [\Delta^k_{(1)}(4;1):L^k_{(4)}(4;1)]&=[\Delta^K_{(1)}(4;1):L^K_{(4)}(4;1)][S^{(4)}_k:D^{(4)}_k]\\
	 &~~~~~+[\Delta^K_{(1)}(4;1):L^K_{(2^2)}(4;1)][S^{(2^2)}_k:D^{(4)}_k]\\
	 &=[\Delta^K_{(1)}(4;1):L^K_{(4)}(4;1)]+[\Delta^K_{(1)}(4;1):L^K_{(2^2)}(4;1)]
	 \end{align*}
	 From Theorem \ref{thm:martinblocks} we know that all non-decomposition numbers in characteristic zero correspond to \linebreak$\delta$-pairs. However neither $(4)$ and $(1)$ nor $(2^2)$ and $(1)$ are $1$-pairs, and therefore both these decomposition numbers are zero. This contradicts \eqref{eq:partlift4}, and the factorisation must in fact not be valid for $r=0$.
	 
	 Now let $r=1$, $\lambda=(2,1^2)$ and $\mu=(2,1)$. Again by examining the decomposition matrix of $k\sym_4$, we see that the only partition $\nu$ for which $[S^\nu_k:D^{(2,1^2)}_k]\neq0$ is $\nu=(2,1^2)$. The factorisation then becomes
	 \begin{align*}
	 [\Delta^k_{(2,1)}(4;1):L^k_{(2,1^2)}(4;1)]&=[\Delta^K_{(2,1)}(4;4):L^K_{(2,1^2)}(4;4)][S^{(2,1^2)}_k:D^{(2,1^2)}_k]\\
	 \\
	 &=[\Delta^K_{(2,1)}(4;4):L^K_{(2,1^2)}(4;4)]
	 \end{align*}
	Again we see that $(2,1^2)$ and $(2,1)$ is not a $4$-pair, and therefore this decomposition number is zero. This contradicts \eqref{eq:partlift1}, and the factorisation is not valid for $r=1$.
	
	Since $\delta=1$ and $\delta=4$ are the only values of $\delta$ such that $P_4^K(\delta)$ is non-semisimple, we see that there is no $r\in\mathbb{Z}$ that allows us to express the decomposition matrix in characteristic $p$ as a product as above.
\end{example}
\subsection*{Acknowledgements}
The author would like to thank Maud De Visscher for her helpful comments and discussions.
\bibliographystyle{amsalpha}
\bibliography{decompositionmatrices}
~\\
\textsc{Centre for Mathematical Science, City University London, Northampton Square, London, EC1V 0HB, United Kingdom}\\
\emph{E-mail address:} \verb+oliver.king.1@city.ac.uk+
\end{document}